\documentclass[11pt]{amsart}
\usepackage{graphicx}
\usepackage{amscd}
\usepackage{amsmath}
\usepackage{amsxtra}
\usepackage{amsfonts}
\usepackage{amssymb}
\usepackage{xcolor}
\usepackage{mathrsfs}

\newtheorem{theorem}{Theorem}[section]
\newtheorem{corollary}[theorem]{Corollary}
\newtheorem{lemma}[theorem]{Lemma}
\newtheorem{proposition}[theorem]{Proposition}

\theoremstyle{definition}
\newtheorem{definition}[theorem]{Definition}

\newtheorem{remark}[theorem]{Remark}

\newtheorem{example}[theorem]{Example}
\theoremstyle{remark}

\renewcommand{\theclaim}{\textup{\theclaim}}

\newtheorem*{acknowledgements}{Acknowledgements}

\numberwithin{equation}{section}

\def\openone

{\mathchoice

{\hbox{\upshape \small1\kern-3.3pt\normalsize1}}

{\hbox{\upshape \small1\kern-3.3pt\normalsize1}}

{\hbox{\upshape \tiny1\kern-2.3pt\SMALL1}}

{\hbox{\upshape \Tiny1\kern-2pt\tiny1}}}

\makeatletter

\newbox\ipbox

\newcommand{\ip}[2]{\left\langle #1\, , \,#2\right\rangle}
\newcommand{\diracb}[1]{\left\langle #1\mathrel{\mathchoice

{\setbox\ipbox=\hbox{$\displaystyle \left\langle\mathstrut
#1\right.$}

\vrule height\ht\ipbox width0.25pt depth\dp\ipbox}

{\setbox\ipbox=\hbox{$\textstyle \left\langle\mathstrut
#1\right.$}

\vrule height\ht\ipbox width0.25pt depth\dp\ipbox}

{\setbox\ipbox=\hbox{$\scriptstyle \left\langle\mathstrut
#1\right.$}

\vrule height\ht\ipbox width0.25pt depth\dp\ipbox}

{\setbox\ipbox=\hbox{$\scriptscriptstyle \left\langle\mathstrut
#1\right.$}

\vrule height\ht\ipbox width0.25pt depth\dp\ipbox}

}\right. }

\newcommand{\dirack}[1]{\left. \mathrel{\mathchoice

{\setbox\ipbox=\hbox{$\displaystyle \left.\mathstrut
#1\right\rangle$}

\vrule height\ht\ipbox width0.25pt depth\dp\ipbox}

{\setbox\ipbox=\hbox{$\textstyle \left.\mathstrut
#1\right\rangle$}

\vrule height\ht\ipbox width0.25pt depth\dp\ipbox}

{\setbox\ipbox=\hbox{$\scriptstyle \left.\mathstrut
#1\right\rangle$}

\vrule height\ht\ipbox width0.25pt depth\dp\ipbox}

{\setbox\ipbox=\hbox{$\scriptscriptstyle \left.\mathstrut
#1\right\rangle$}

\vrule height\ht\ipbox width0.25pt depth\dp\ipbox}

} #1\right\rangle}

\newcommand{\beq}{\begin{equation}}

\newcommand{\eeq}{\end{equation}}

\newcommand{\cj}[1]{\overline{#1}}

\newcommand{\bz}{\mathbb{Z}}

\newcommand{\br}{\mathbb{R}}
\newcommand{\bc}{\mathbb{C}}

\newcommand{\bn}{\mathbb{N}}

\def\blfootnote{\xdef\@thefnmark{}\@footnotetext}

\renewcommand{\mod}{\operatorname{mod}}

\input xy
\xyoption{all}
\usepackage{amssymb}

\newcommand{\supp}{\operatorname*{supp}}

\def\a{\mathfrak{a}}

\def\-{^{-1}}

\def\D{\mathscr{D}}

\def\ty{\emptyset}

\begin{document}

\title[Fuglede's conjecture and differential operators]{Fuglede's conjecture, differential operators and unitary groups of local translations}

\author{Piyali Chakraborty}
\address{[Piyali Chakraborty] University of Central Florida\\
	Department of Mathematics\\
	4000 Central Florida Blvd.\\
	P.O. Box 161364\\
	Orlando, FL 32816-1364\\
	U.S.A.\\} \email{Piyali.Chakraborty@ucf.edu}

\author{Dorin Ervin Dutkay}
\address{[Dorin Ervin Dutkay] University of Central Florida\\
	Department of Mathematics\\
	4000 Central Florida Blvd.\\
	P.O. Box 161364\\
	Orlando, FL 32816-1364\\
U.S.A.\\} \email{Dorin.Dutkay@ucf.edu}

\author{Palle E.T. Jorgensen}
\address{[Palle E.T. Jorgensen]University of Iowa\\
Department of Mathematics\\
14 MacLean Hall\\
Iowa City, IA 52242-1419\\}\email{palle-jorgensen@uiowa.edu}


\dedicatory{Dedicated to the memory of Professor Bent Fuglede.}

\begin{abstract}

The purpose of the present paper is to address multiple aspects of the Fuglede question dealing (Fourier spectra vs geometry) with a variety of $L^2$ contexts where we make precise the interplay between the three sides of the question: (i) existence of orthogonal families of Fourier basis functions (and associated spectra) on the one hand, (ii) extensions of partial derivative operators, and (iii) geometry of the corresponding domains, stressing systems of translation-tiles. We emphasize an account of old and  new  developments since the original 1974-paper by Bent Fuglede where the co-authors and Steen Pedersen have contributed.
 
\end{abstract}
\maketitle
\tableofcontents
\newcommand{\Ds}{\mathsf{D}}
\newcommand{\Dmax}{\mathscr D_{\operatorname*{max}}}
\newcommand{\Dmin}{\Ds_{\operatorname*{min}}}
\newcommand{\dom}{\operatorname*{dom}}

\section{Introduction}

\bibliographystyle{alpha}

We recall that, in its original form, the Fuglede question (see \cite{Fug74}) dealing with spectrum
vs geometry, addressed only the special case of open and bounded subsets $\Omega$ in $\br^n$.
We recall that in this context, Fuglede's question/conjecture is about the interconnection
between two properties that such sets $\Omega$ may or may not have, one referring to a notion
of spectrum, and the other to geometry. More specifically, the spectral part is the question is
when $L^2(\Omega)$ allows an orthogonal basis of Fourier frequencies, and the other (ii) asks
when $\Omega$ arises as a tile for some set of translation vectors in $\br^n$. In summary, we say
that $\Omega$ is spectral if (i) holds, and is tile in case (ii). The two, (i) and (ii), are known to be
equivalent if $\Omega$ is further assumed convex \cite{LM22}.

Even though Terence Tao et al (see \cite{Tao04, FMM06}) proved that the two (i) and (ii) are not equivalent, in dimension $n>2$, (neither implies the other), the study of a variety of related frameworks and
themes, dealing with connections between spectrum and geometry has blossomed, to now
include e.g., the case of more general classes of $L^2$-measure spaces  $L^2(\mu)$. And of course,
spectrum vs geometry is central in mathematics. Currently, by now the study of diverse classes
of choices of fractal, or self-similar, measures $\mu$  is actively pursued. And yet, in other studies,
the orthogonality property for the Fourier-frequencies is relaxed, e.g., by requiring instead only a
frame-property relative to the $L^2(\mu)$ space under consideration.
In fact, by now there are many diverse and very active research communities which study
such much more general formulations of spectrum vs geometry. This includes the present co-authors, see \cite{Jor18,DJ23,DJ15a,DJ15b,DJ13a,DJ13b,DHJ13,DJ12a,DJ12b,DHJ09,DJ08,DJ07a,DJ07b}, Jorgensen-Pedersen \cite{JPT15, JP99, JP98a, JP98b, JP93, JP92, JP91},
Iosevich et al \cite{IKP99, IK13, IR03, IMP17}, Nir Lev et al \cite{EL23, Lev22, LM22, DL19, GL20, LO17,
KL16}, to mention just a few. Recent papers by Jorgensen-Herr-Weber \cite{HJW23, HJW20,
HJW19a, HJW19b, HJW18} aim for a classification of non-atomic measures $\mu$ with compact
support in $\br^n$, such that $L^2(\mu)$ admits a (generally non-orthogonal) Fourier expansion. The
2018 CBMS-volume \cite{Jor18} by Jorgensen offers an (partial) overview.

  We should add that, for the case of infinite convolutions and random convolutions, Hadamard triples, and selfsimilar fractals,  the literature is vast. It addresses both issues of spectral duality, as well as related questions, its determination, and its implications. Here we cite the following papers in the general area, but there are many more, see \cite{LZ24, CDL24, LJW24, LW24, Liu24, YZ24, WYZ24, LMW24} and the papers cited there.

Our present paper focuses on the case of open domains $\Omega$ in $\br^n$ and the Hilbert space $L^2(\Omega)$. Here we then introduce the respective first order partial derivative operators in the $n$ coordinate directions. As operators in $L^2(\Omega)$, they are unbounded, but we point out that they have canonical minimal and maximal realizations. In this context, we then study how the Fuglede problem may be formulated in terms of choices of strongly continuous unitary representations $U$ of $\br^n$, acting on $L^2(\Omega)$. With the use of the Stone Theorem for unitary representations of the group $\br^n$ we show that the Fourier spectrum from the Fuglede question may be derived from the spectrum for $U$. We then show that the solution to the Fuglede question takes the form of an identification of certain systems of $n$  commuting self-adjoint operators, with each one arising as a restriction of the corresponding maximal $L^2(\Omega)$ partial derivative operator; or equivalent as extensions of the corresponding minimal operator. We note that such commuting self-adjoint operators do not always exist, but the existence (when affirmative) of such commuting systems of self-adjoint extension operators is an important question in its own right, and it throws new light on the classical Fuglede problem.

We will start from the origins of the Fuglede conjecture: Segal's question about the existence of commuting self-adjoint restrictions of the partial differential operators.

{\bf Segal's question.} As mentioned in \cite{Fug74}, in 1958 Irving Segal posed the following problem to Fuglede: Which are the connected open sets $\Omega\subset \br^n$ such that there exist, on the Hilbert space $L^2(\Omega)$, commuting self-adjoint restrictions $H_1,\dots,H_n$ of the partial differential operators $D_1=\frac1{2\pi i}\frac\partial{\partial x_1},\dots,D_n=\frac1{2\pi i}\frac\partial{\partial x_n}$ (acting on $L^2(\Omega)$ in the distribution sense), the commutation being understood in the sense of commuting spectral measures.

Fuglede proved in \cite{Fug74} that, for connected open sets $\Omega\subset\br^n$, of finite Lebesgue measure, satisfying a mild regularity condition (finiteness of the Poincar\' e constant, or Nikodym domains, see the Appendix), a necessary and sufficient condition is that there should exist a set $\Lambda\subset \br^n$ such that the functions $e_\lambda(x)=e^{2\pi i\lambda\cdot x}$, $\lambda\in\Lambda$, form a complete orthogonal family in $L^2(\Omega)$ (Theorem \ref{th2.1}); he called such sets {\it spectral}. He then formulated his famous conjecture: a set $\Omega$ is spectral if and only if it tiles $\br^n$ by translations. Fuglede also proved that simple domains like a disk or a triangle are not spectral. Steen Pedersen improved Fuglede's result in \cite{Ped87} to include general connected open sets without the Nikodym restriction, and even such sets of infinite measure.  

In tribute to Professor Fuglede’s captivating expository style, we present a detailed account of his and Steen Pedersen’s beautiful arguments. This paper was written with our students in mind, as we believe Professor Fuglede's work serves as an excellent illustration of key principles in spectral theory. 

In Section \ref{secfu}, we present Fuglede's proof of Theorem \ref{th2.1} and in Section \ref{secre} we present an example that shows why the Nikodym condition is not easy to remove.

In Section \ref{secpe} we present Pedersen's extensions of Theorem \ref{th2.1} from \cite{Ped87}. Steen Pedersen improved Fuglede’s result, not only by removing the Nikodym restriction but also allowing the domain to have infinite measure.

When there are commuting self-adjoint restrictions $\{H_i\}$ of the differential operators $\{D_i\}$ they have a joint spectral measure $E$ and one can construct a unitary group of operators $\{U(t)\}$ (Theorem \ref{thstone0}). The converse is also true, by the Generalized Stone Theorem (Theorem \ref{thgs}). Since the operators $\{H_i\}$ are restrictions of the differential operators $\{D_i\}$, the unitary group $U$ acts locally as translations (Definition \ref{defp5}). Conversely , if the unitary group $U$ acts as local translations then the self adjoint operators $\{H_i\}$ are indeed restrictions of the differential operators $\{D_i\}$ (Theorem \ref{thai}). Pedersen clarified the equivalence between (1) spectral sets, (2) the existence of  commuting self-adjoint restrictions $\{H_i\}$, and (3) the existence of unitary groups $U$ that act as local translations (Proposition \ref{prp6}, Theorem \ref{thai}), for connected open sets.

In section \ref{secun} we present some of our work on the one dimensional case when the set  $\Omega$ is a finite union of intervals. In this instance, more details can be provided for the domain of the self-adjoint restriction $H$ and for its unitary group $U$.

In section \ref{secex} we provide some examples that enhance the understanding and offer intuition for the results.

We use the last section as an Appendix where we list some of the fundamental results used in the arguments: various forms of the Spectral Theorem and of Stone's Theorem,  some important inequalities such as the Poincar\'e inequality and an Ehrling-Sobolev inequality, and a few others.

\section{Preliminaries}\label{secpre}

We denote by $m$ the Lebesgue measure on $\br^n$. For any open subset $\Omega$ of $\br^n$ with $0<m(\Omega)<\infty$ we use 
$$\ip{f}{g}=\int_\Omega f(x)\cj g(x)\,dx$$
as the inner product for the complex Hilbert space $L^2(\Omega)$. 

For $\lambda=(\lambda_1,\dots,\lambda_n)\in\br^n$ and $x=(x_1,\dots,x_n)\in\br^n$, we denote
$$e_\lambda(x)=e^{2\pi i\lambda\cdot x},\quad \lambda\cdot x=\sum_{i=1}^n\lambda_ix_i.$$

The differential operators 
$$D_j=\frac{1}{2\pi i}\frac\partial{\partial x_j},\quad(j=1,\dots, n)$$
are understood in general to act in the distribution sense on $L^2(\Omega)$. Thus $D_j$ is the ``maximal'' operator on $L^2(\Omega)$ with the domain 
$$\mathscr D(D_j)=\{u\in L^2(\Omega) : D_j u\in L^2(\Omega)\}.$$

On a complex Hilbert space $\mathfrak H$, let 
$$H=(H_1,\dots,H_n)$$
denote a family of $n$ self-adjoint (not necessarily bounded) operators which commute with one another, in the sense of commuting spectral measures. The domain of this family is 
$$\mathscr D(H)=\bigcap_{j=1}^n \mathscr D( H_j).$$

Denoting by $\xi=(\xi_1,\dots,\xi_n)$ the generic point of $\br^n$, we have the canonical spectral representation 
$$H=\int\xi\,dE,$$
that is 

\begin{equation}\label{eqE}
H_j=\int \xi_j\,dE,\quad (j=1,\dots,n),
\end{equation}

of the family $H$. 

Here $E$ denotes the spectral measure on $\br^n$ associated with $H$ (see Theorem \ref{thspectral}). It is defined from the spectral measures $E_i$ of the self-adjoint operators $H_i$ by 

\begin{equation}\label{eqprE}
E(\Delta_1\times\Delta_2\times\dots\times\Delta_n)=E_1(\Delta_1)E_2(\Delta_2)\dots E_n(\Delta_n),\mbox{ for }\Delta_1,\dots,\Delta_n\mbox{ measurable  in }\br.
\end{equation}

The support of $E$ is the {\it (simultaneous/ joint) spectrum} of $H$ and is denoted by $\sigma(H)$. It is a closed subset of $\br^n$. The atomic part of $E$ determines the {\it point spectrum} $\sigma_p(H)$ of $H$. It consists of all points $\lambda\in\br^n$ such that $E(\{\lambda\})\neq 0$, in other words, of all (simultaneous) {\it eigenvalues} for $H$, an eigenvalue $\lambda=(\lambda_1,\dots,\lambda_n)$ for $H=(H_1,\dots,H_n)$ being a point of $\br^n$ such that the subspace ({\it eigenspace})
$$E(\{\lambda\})=\{u\in\mathscr D(H) : Hu=\lambda u\}$$
is not $\{0\}$. (The relation $Hu=\lambda u$ means $H_ju=\lambda_j u$ for every $j=1,\dots, n$.)

\begin{remark}\label{rem2.1}
	There is another way to look at the self-adjoint restrictions of the maximal differential operators $D_i$: they are also self-adjoint {\it extensions} of the ``minimal'' differential operator $\frac{1}{2\pi i}\frac{\partial}{\partial x_i}$ on $C_0^\infty(\Omega)$ - the space of compactly supported infinitely differentiable functions on $\Omega$. 
	
	Indeed, if $H_i$ is a self-adjoint restriction of the maximal operator $D_i$, then, for $f\in C_0^\infty(\Omega)$, the linear functional 
	$$C_0^\infty(\Omega)\ni g\rightarrow\ip{H_if}{g}=(D_if,g)=(f,D_ig)=-\int_\Omega f(x)\frac{1}{2\pi i}\frac{\partial\cj g}{\partial x_i}(x)\,dx=\int_\Omega \frac{1}{2\pi i}\frac{\partial f}{\partial x_i}(x)\cj g(x)\,dx.$$
	extends to a continuous linear functional on $L^2(\Omega)$, and therefore $f\in\mathscr D(H_i^*)=\mathscr D(H_i)$, 
	and $H_if=H_i^*f=\frac{1}{2\pi i}\frac{\partial f}{\partial x_i}(x).$ This means that $H_i$ is an extension of the minimal operator. 
	
	We use the notation $$(f,\varphi)=f(\cj\varphi),\mbox{ for $\varphi\in C_0^\infty(\Omega)$ and a distribution $f$ on $\Omega$.}$$ 
	
	Conversely, if $H_i$ is a self-adjoint extension of the minimal operator  $\frac{1}{2\pi i}\frac{\partial}{\partial x_i}$ on $C_0^\infty(\Omega)$, then, for $f\in \mathscr D(H_i)=\mathscr D(H_i^*)$, and $g\in C_0^\infty(\Omega)$, 
	$$ \ip{ H_i f}{g}=\ip{f}{H_ig}=-\frac{1}{2\pi i}\int_\Omega f(x) \frac{\partial \cj g}{\partial x_i}(x)\,dx,$$
	which means that $H_if$ has a weak derivative $D_if$ equal to $H_if$. This means that $H_i$ is a restriction of the maximal operator $D_i$.

\end{remark}

\begin{definition}
	\label{defsp}
	A measurable set $\Omega$ in $\br^n$  with finite measure is called spectral if there exists an index set $\Lambda\subset\br^n$ such that the family of exponential functions $\{e_\lambda :\lambda\in\Lambda\}$ is an orthogonal basis for $L^2(\Omega)$.
\end{definition}

\section{Bent Fuglede: Spectral sets and joint spectra}\label{secfu}

The next result of Fuglede is the main focus of our paper. It gives an answer to Segal's question: under some mild restrictions (bounded Nikodym domain, Definition \ref{defnik}) there are self-adjoint restrictions of the partial differential operators if and only if the set $\Omega$  is spectral, i.e., there is an orthogonal basis of exponential functions for $L^2(\Omega)$. 

We present here Fuglede's proof and we add some motivation and underline the key steps in the argument. 

\begin{theorem}\label{th2.1}\cite[Theorem I]{Fug74}	Let $\Omega\subset\br^n$ be a finite measure open and connected Nikodym region (see Definition \ref{defnik}).
	
	\begin{enumerate}
		\item[(a)] Let $H=(H_1,\dots,H_n)$ denote a commuting family (if any) of self-adjoint restrictions $H_j$ of $D_j$ on $L^2(\Omega)$, $j=1,\dots,n$. Then $H$ has a discrete spectrum, each point $\lambda\in \sigma(H)$ being a simple eigenvalue for $H$ with the eigenspace $\bc e_\lambda$, and hence $(e_\lambda)_{\lambda\in \sigma(H)}$ is an orthogonal basis for $L^2(\Omega)$. Moreover, $\sigma(H)=\sigma_p(H)=\{\lambda\in\br^n : e_\lambda\in\mathscr D( H)\}$.
		\item[(b)] Conversely, let $\Lambda$ denote a subset (if any) of $\br^n$ such that $(e_\lambda)_{\lambda\in\Lambda}$ is an orthogonal basis for $L^2(\Omega)$. Then there exists a unique commuting family $H=(H_1,\dots,H_n)$ of self-adjoint restrictions $H_j$ of $D_j$ on $L^2(\Omega)$ with the property that $\{e_\lambda : \lambda\in\Lambda\}\subset\mathscr D(H)$, or equivalently that $\Lambda=\sigma(H)$.
	\end{enumerate}
\end{theorem}

We start in a general setting: 	on a complex Hilbert space $\mathfrak H$, let $$H=(H_1,\dots,H_n)$$
denote a family of $n$ self-adjoint (not necessarily bounded) operators which commute with one another, in the sense of commuting spectral measures.

We want to show that, for $\lambda$ in the spectrum $\sigma(H)$, the spectral projection $E(\{\lambda\})$ is nonzero. For this, we consider a sequence of balls $$B_{1/p}(\lambda)=\{\xi\in\br^n : \|\xi-\lambda\|<1/p\},\quad (p=1,2,\dots),$$
centered at $\lambda$ and with radius converging to zero. 

Since $\lambda$ is in the spectrum $\sigma(H)$, it follows that $E(B_{1/p}(\lambda))\neq 0$ (see Theorem\ref{thspectral}(b)), and therefore we can choose $u_p\in E(B_{1/p}(\lambda))\mathfrak H$ so that $\|u_p\|=1$; note that $u_p\in \mathscr D( H)$, because 
$$\int |\xi|^2\,d\ip{E(\xi)u_p}{u_p}=\int |\xi|^2\,d\ip{E(\xi)E(B_{1/p}(\lambda))u_p}{E(B_{1/p}(\lambda))u_p}=$$$$\int_{B_{1/p}(\lambda)}|\xi|^2\,d\ip{E(\xi)u_p}{u_p}\leq (\|\lambda\|+1/p)^2<\infty,$$
since $\|\xi\|\leq \|\xi-\lambda\|+\|\lambda\|<\|\lambda\|+1/p$, for $\xi\in B_{1/p}(\lambda)$, and since $\|u_p\|=1$ (see also Remark \ref{remspec}).

We also have 
$$\sum_{j=1}^n\|H_ju_p-\lambda_j u_p\|^2=\sum_{j=1}^n\|(H_j-\lambda_jI) E(B_{1/p}(\lambda))u_p\|^2=\int_{B_{1/p}(\lambda)}\|\xi-\lambda\|^2\,d\ip{E(\xi)u_p}{u_p}\leq 1/p^2.$$

It follows that, for any $j=1,\dots,n$
\begin{equation}
	H_ju_p-\lambda_ju_p\rightarrow0\mbox{ strongly in }\mathfrak H\mbox{ as }p\rightarrow\infty.
	\label{eql2.2.2}
\end{equation}

Since $\|u_p\|=1$ for all $p\in\bn$, passing if necessary, to a suitable subsequence, we may assume that there is a vector $u\in\mathfrak H$ such that 
\begin{equation}
	u_p\rightarrow u\mbox{ weakly in }\mathfrak H.
\end{equation}

The key question is: how can we get that $u_p\rightarrow u$ strongly? 

Let's say we obtain that. Then $\lambda_j u_p\rightarrow \lambda_j u$ strongly, and using \eqref{eql2.2.2}, we have that $H_ju_p\rightarrow \lambda_j u$ strongly. Since $u_p\rightarrow u$, and since $H_j$ is self-adjoint and therefore closed, we get that $u$ is in $\mathscr D( H_j)$ and $H_ju=\lambda_j u$. Also, $\|u_p\|=1$ for all $p$, and it follows that $\|u\|=1$ hence it is a non-zero eigenvector.

Hence, the question is: if we know that $H_ju_p-\lambda_j u_p\rightarrow 0$ strongly and $u_p\rightarrow  u$ weakly, what extra conditions do we need to obtain that $u_p\rightarrow u$ strongly? 

Here is one of the main ideas in Fuglede's paper; we will try to explain what is the intuition behind it. In the particular case when our operators $H_j$ are commuting self-adjoint restrictions of the differential operators $D_j$, the conditions $H_ju_p-\lambda_ju_p\rightarrow 0$ strongly tell us something about the partial derivatives. However, if we know all the partial derivatives of a function in a connected domain, then we know the function, up to an additive constant. But we need a more quantitative version of this statement: the $L^2$-convergence of the partial derivatives should imply some $L^2$-convergence of the function (if we ignore the additive constant). This is given to us by the Poincar\'e inequality \eqref{eqPoincare}! We will use the second version \eqref{eqPoincare2} (see also Definition \ref{defnik}).

Fuglede formulates a very nice, more general lemma where this argument can be used. Here it is.

\begin{lemma}\label{lem2.2}\cite[p. 104]{Fug74}	On a complex Hilbert space $\mathfrak H$, let 
	$$H=(H_1,\dots,H_n)$$
	denote a family of $n$ self-adjoint (not necessarily bounded) operators which commute with one another, in the sense of commuting spectral measures. 
	Let $P$ denote a finite dimensional orthogonal projection operator on $\mathfrak H$, and let $\lambda=(\lambda_1,\dots,\lambda_n)$ denote a given point of the spectrum $\sigma(H)$. Suppose that there exists a finite constant $C$ such that 
	\begin{equation}\label{eq2.2.1}
		\|u-Pu\|^2\leq C\sum_{j=1}^n\|H_ju-\lambda_ju\|^2
	\end{equation}
	for all $u\in\mathscr D(H)$. Then 
	$$0\neq E(\{\lambda\})\leq P,$$
	that is, $\lambda$ is an eigenvalue for $H$ with eigenspace $ E(\{\lambda\})$ contained in the range of $P$. If $E(\{\lambda\})=P$ then $\lambda$ is an isolated point of $\sigma(H)$, the distance between $\lambda$ and $\sigma(H)\setminus\{\lambda\}$ being $\geq1/C^{1/2}$.
	
\end{lemma}

\begin{proof}
	The argument begins as above, for $\lambda\in\sigma(H)$, using the spectral projections $E(B_{1/p(\lambda)})$ we obtain a sequence $u_p\in\mathfrak H$, $\|u_p\|=1$, such that $H_ju_p-\lambda_ju_p\rightarrow 0$ for all $j=1,\dots,n$ and $u_p\rightarrow u$ weakly. Then, using \eqref{eq2.2.1}, we have that 
	\begin{equation}
		u_p-Pu_p\rightarrow 0\mbox{ strongly in }\mathfrak H.
		\label{eq2.2.3}
	\end{equation}
	
	Since $u_p\rightarrow u$ weakly, $Pu_p\rightarrow Pu$ weakly, and since $P$ is finite dimensional, it follows that $Pu_p\rightarrow Pu$ strongly. Combining this with \eqref{eq2.2.3} we obtain that $u_p\rightarrow Pu$ strongly, hence weakly, and therefore $Pu=u$. 
	
	Following the argument presented just before Lemma \ref{lem2.2}, we get that $u\in\mathscr D( H)$, $H_ju=\lambda_j u$ for each $j=1,\dots,n,$ $\|u\|=1$ and $E(\{\lambda\})u=u$, so $E(\{\lambda\})\neq 0$. 
	
	For any $u\in E(\{\lambda\})\mathfrak H$ we have $H_ju=\lambda_ju$ for every $j=1,\dots,n$, and hence, it follows from \eqref{eq2.2.1} that $u=Pu$, that is $u\in P\mathfrak H$.

	Suppose now, in addition to \eqref{eq2.2.1}, that $E(\{\lambda\})=P$, and let
	$$B^*_\rho=\{\xi\in\br^n : 0<\|\xi-\lambda\|<\rho\},$$
	for some fixed $\rho$ such that $0<\rho<1/C^{1/2}$. For any $u\in E(B_\rho^*)\mathfrak H$ we obtain $u\in\mathscr D (H)$ and 
	\begin{equation}
	\sum_{j=1}^n\|H_ju-\lambda_ju\|^2=\int_{B_\rho^*}\|\xi-\lambda\|^2\,d\ip{E(\xi)u}{u}\leq \rho^2\|u\|^2.
	\label{eq2.2.4}
	\end{equation}
	
	On the other hand, 
	$$\|u-Pu\|^2=\|u\|^2,$$
	since $Pu=E(\{\lambda\})u=0$, because $\lambda\not\in B_\rho^*$, and therefore the corresponding spectral projections are orthogonal. It follows now from \eqref{eq2.2.1} and \eqref{eq2.2.4} that $u=0$, because $\rho^2<1/C$. Consequently $E(B_\rho^*)=0$ and so $B_\rho^*\cap \sigma(H)=\ty$, which means that the distance from $\lambda$ to $\sigma(H)\setminus\{\lambda\}$ is at least $1/C^2$. 
\end{proof}

With Lemma \ref{lem2.2}, the proof of Theorem \ref{th2.1} is quite easy. 

\begin{proof}[Proof of Theorem \ref{th2.1}]
	
For part (a), take $H_j$ to be the self-adjoint restrictions of the differential operators $D_j$, and let $\lambda\in\sigma(H)$. Poincar\'e's inequality \eqref{eqPoincare2} shows that we can take $P=P_\lambda$ in Lemma \ref{lem2.2}, where $P_\lambda$ is the one-dimensional projection $P_\lambda u=\frac{1}{m(\Omega)}\ip{u}{e_\lambda}e_\lambda$ for every $u\in L^2(\Omega)$. It follows that 
	$$E(\{\lambda\})=P_\lambda,\quad \mbox{for every }\lambda\in\sigma(H),$$
	because the range $\bc e_\lambda$ of $P_\lambda$ is one-dimensional, and non-zero. By the final assertion of Lemma \ref{lem2.2}, the spectrum $\sigma(H)$ is discrete, and hence $H$ has a pure point spectrum, and $\sigma(H)=\sigma_p(H)$. 
	
	Thus, we have the following implications: if $\lambda\in\sigma(H)$ then $e_\lambda=P_\lambda e_\lambda=E(\{\lambda\})e_\lambda \in \mathscr D( H)$. If $e_\lambda\in \mathscr D( H)$ then $He_\lambda=\lambda e_\lambda$ so $\lambda\in\sigma(H)$. Therefore 
	$$\sigma(H)=\{\lambda\in\br^n : e_\lambda\in\mathscr D(H)\}=\{\lambda\in\br^n: He_\lambda=\lambda e_\lambda\},$$
	and since $\sum_{\lambda\in \sigma(H)}E(\{\lambda\})=I_{\mathfrak H}$, it follows that $\{e_\lambda : \lambda\in \sigma(H)\}$ is an orthogonal basis for $\mathfrak H$.

	For part (b), let $\Lambda$ denote a subset of $\br^n$ such that $(e_\lambda)_{\lambda\in\Lambda}$ is an orthogonal basis for $L^2(\Omega)$. Define 
	$$\mathscr D( H_j)=\left\{\sum_{\lambda\in\Lambda}c_\lambda e_\lambda : \sum_{\lambda\in\Lambda}(1+|\lambda_j|^2)|c_\lambda|^2<\infty\right\}, \mbox{ for }j=1,\dots n,$$
	and 
	$$H_j\sum_{\lambda\in\Lambda}c_\lambda e_\lambda =\sum_{\lambda\in\Lambda}\lambda_jc_\lambda e_\lambda .$$
	
	If some vector $v=\sum_{\lambda\in\Lambda }v_\lambda e_\lambda$ is in the domain of the adjoint $\mathscr D(H_j^*)$, then the linear functional $\mathscr D( H_j)\ni u\mapsto \ip{H_ju}{v}$ must be bounded, so there exists a constant $A\geq 0$ such that 
	$$\left|\sum_{\lambda\in\Lambda}\lambda_jc_\lambda \cj v_\lambda\right|\leq  A\left(\sum_{\lambda\in\Lambda}|c_\lambda|^2\right)^{1/2},\mbox {for all finitely supported sequences }(c_\lambda)_{\lambda\in\Lambda}.$$
	But this implies that 
	$$\left(\sum_{\lambda\in\Lambda}|\lambda_j|^2|v_\lambda|^2\right)^{1/2}\leq A^2<\infty,$$
	and this means that $v\in\mathscr D( H_j)$. Consequently, $\mathscr D( H_j^*)\subseteq \mathscr D(H_j)$ and $H_j$ is self-adjoint. Since $H_je_\lambda=\lambda_je_\lambda=D_je_\lambda$, it follows that $H_j$ is a restriction of the maximal operator $D_j$.
	
	Note also that, in the basis $(e_\lambda)_{\lambda\in \Lambda}$, the operators $H_j$ have diagonal infinite matrices, therefore, the same will be true for their spectral projections. Hence, the operators $H_j$ have commuting spectral measures. 
	
	For uniqueness, if $H$ is a family of commuting self-adjoint restrictions $H_j$ of $D_j$, with $\{e_\lambda :\lambda\in\Lambda\}\subset\mathscr D(H)$. Since $\{e_\lambda :\lambda\in\Lambda\}$ is an orthogonal basis, $H_j$ has to be the closure of its restriction $H_j^0$ to the subspace algebraically spanned by $\{e_\lambda : \lambda\in\Lambda\}$; also $H_je_\lambda=D_je_\lambda=\lambda_je_\lambda$. Therefore $H_j$ is uniquely determined from $\Lambda$. 
\end{proof}

\begin{remark}\label{rem2.3}
 Note that $\Omega$ does not have to be connected nor a Nikodym set for part (b). 
\end{remark}

\section{More on the Nikodym restriction}\label{secre}
In this section we present an example that shows why the Nikodym restriction is not easy to remove. We will build a bounded connected open set $\Omega$ in $\br^2$ and a sequence of functions $\{u_p\}$ in $C^\infty(\Omega)$ with the following properties 
\begin{enumerate}
	\item $\|u_p\|=1$ for all $p$.
	\item $u_p\rightarrow 0$ weakly in $L^2(\Omega)$.
	\item $D_ju_p\rightarrow 0$ strongly in $L^2(\Omega)$ for $j=1,2$.
	\end{enumerate}

This will show that the Poincar\'e inequality does not hold on $\Omega$, and the arguments used in the previous section for the proof of Fuglede's Theorem \ref{th2.1} do not work for this domain.

We define the set $\Omega$ in $\br^2$ as follows. First let $S_n=(x_n,x_n+s_n)\times(0,s_n)$ be a sequence of disjoint open squares with bottom side on the $x$-axis, $0=x_1<x_1+s_1<x_2<x_2+s_2<\dots$. We will adjust the numbers $\{x_n\}$ and $\{s_n\}$ later. The squares are smaller and smaller, and we will require 
$\sum_n s_n<\infty$. 

We connect two consecutive squares $S_n$, $S_{n+1}$ with a very thin band $B_n=[x_n+s_n,x_{n+1}]\times (0,\delta_n)$, $\delta_n>0$. 

The set $\Omega$ is 
$$\Omega=\bigcup_{n=1}^\infty (S_n\cup B_n).$$

First, we pick the sides of the squares to be very small $s_n=\frac1{2^n}$. Also, we pick the lengths of the bands to be $l_n=x_{n+1}-(x_n+s_n)=\frac1{2^n}$. Note that $\Omega$ is bounded, open and connected. 

Now we construct the functions $\tilde u_p$, $p=1,2,\dots$, and then we normalize $u_p=\frac{1}{\|\tilde u_p\|}\tilde u_p$. The function $\tilde u_p$ will be supported on the square $S_{4p}$ together with the connected bands $B_{4p-1}$ and $B_{4p}$. We use $4p$ to make sure that the supports of the functions $u_p$ are disjoint.

On the square $S_{4p}$ the functions $\tilde u_p$ will be constant $c_p={2^{4p}}$ (to make the $L^2$-norm on $S_{4p}$ equal to 1). On all the other squares, the function $\tilde u_p$ will be $0$. To make the function $\tilde u_p$ differentiable, we use the bands $B_{4p-1}$ and $B_{4p}$ to connect smoothly the values on the square $S_{4p}$ to the zero values on the squares $S_{4p-1}$ and $S_{4p+1}$. For this,  we consider a $C^\infty$-function  $f_p:\br\rightarrow[0,c_p]$ such that 
$$f_p(x)=\left\{\begin{array}{cc}
	0,&\mbox{ if }x<x_{4p-1}+s_{4p-1},\\
	\mbox{ smooth on the band $B_{4p-1}$},&\mbox{ if }x\in[x_{4p-1}+s_{4p-1},x_{4p}],\\
	c_p,&\mbox{ if }x\in[x_{4p},x_{4p}+s_{4p}],\\
	\mbox{ smooth on the band $B_{4p}$},&\mbox{ if }x\in[x_{4p}+s_{4p},x_{4p+1}],\\
	0,&\mbox{if }x>x_{4p}. 
\end{array}\right.
$$.

The function $f_p$ should increase from $0$ to $c_p$ on the interval $[x_{4p-1}+s_{4p-1},x_{4p}]$ of length $l_{4p-1}=\frac{1}{2^{4p-1}}$. This gives an average slope of $2^{4p-1}c_p=2^{8p-1}$. We can make sure that the slope of $f_p$ is close to this value on $L^2$-average (at least, not more than double), so  
\begin{equation}\label{eqc.1}
	\int_{x_{4p-1}+s_{4p-1}}^{x_{4p}}|f_p'(x)|^2\,dx\leq \int_{x_{4p-1}+s_{4p-1}}^{x_{4p}} (2\cdot 2^{8p-1})^2\,dx=\frac{1}{2^{4p-1}}\cdot 2^{16p}=2^{12p+1}.
\end{equation}
Similarly, we can make sure that 
\begin{equation}\label{eqc.2}
	\int_{x_{4p}+s_{4p}}^{x_{4p+1}}|f_p'(x)|^2\,dx\leq\int_{x_{4p}+s_{4p}}^{x_{4p}}(2\cdot 2^{8p})^2\,dx=\frac{1}{2^{4p}}\cdot 2^{16p+2}=2^{12p+2}.  
\end{equation}

We define $\tilde u_p(x,y)=f_p(x)$ for $(x,y)$ in $B_{4p-1}$ or $B_{4p}$. Finally, define $u_p=\alpha_p \tilde u_p$, where $\alpha_p:=\frac{1}{\|\tilde u_p\|}$, so that $\|u_p\|=1$. 

We have 
$$\frac{1}{\alpha_p^2}=\|\tilde u_p\|^2\geq \int_{S_{4p}}|\tilde u_p|^2=c_p^2\cdot s_{4p}^2=1.$$

For condition (iii), note that $\tilde u_p$ is constant on the vertical direction and therefore $D_2u_p=0$, so that is not a problem; we focus on $D_1u_p$. This is supported on the closure of $B_{4p-1}\cup B_{4p}$. We have, using the fact that $\alpha_p\leq 1$, \eqref{eqc.1} and \eqref{eqc.2}, 
$$\|D_1u_p\|^2=\alpha_p^2\left(\int_{B_{4p-1}}|D_1\tilde u_p|^2+\int_{B_{4p}}|D_1\tilde u_p|^2\right)$$$$\leq \delta_{4p-1}\int_{x_{4p-1}+s_{4p-1}}^{x_{4p}}|f_p'(x)|^2\,dx+\delta_{4p}\int_{x_{4p}+s_{4p}}^{x_{4p+1}}|f_p'(x)|\,dx
\leq \delta_{4p-1}\cdot 2^{12p+1}+\delta_{4p}\cdot 2^{12p+2}.$$
We pick $\delta_n$ to be very small to make $\|D_2u_p\|^2\rightarrow0$. For example, let $\delta_n=\frac1{2^{10n}}$. 

It remains only to show that $u_p\rightarrow 0$ weakly in $L^2(\Omega)$. Let $g\in L^2(\Omega)$ and let $\epsilon>0$. There exists $N$ large enough so that 
$$\int_{\cup_{p\geq n} (S_p\cup B_p)} |g|^2<\epsilon,$$
for $n\geq N$.

Then, for $p\geq N$, with the Schwarz inequality, 
$$|\ip{u_p}{g}|^2\leq \left(\int _{B_{4p-1}\cup S_{4p}\cup B_{4p}}|u_p||g|\right)^2\leq \|u_p\|^2\cdot\int_{B_{4p-1}\cup S_{4p}\cup B_{4p} }|g|^2<\epsilon.$$
	Therefore $u_p\rightarrow0$ weakly. 

\section{Steen Pedersen: Unbounded spectral sets,  generalized eigenvectors and unitary groups}\label{secpe}

In this section we present Steen Pedersen's improvements of Fuglede's Theorem \ref{th2.1} from \cite{Ped87}. He not only removed the Nikodym restriction, but also generalized the Theorem to connected domains of infinite measure! 

Of course, if $\Omega$ is of infinite measure then the exponential functions $e_\lambda$ are not in $L^2(\Omega)$ and therefore the definition of spectral sets have to be adjusted. Here is Pedersen's version of spectral sets. 

\begin{definition}\label{defp1}
	For a function $f\in L^1(\br^n)$, let $\hat f$ be the classical Fourier transform 
	$$\hat f(\lambda)=\int_{\br^n}f(t) e^{-2\pi i\lambda\cdot t}\,dt,\quad(\lambda\in\br^n).$$

	Let $\Omega$ be a measurable subset of $\br^n$ and let $\mu$ be a regular positive Borel measure on $\br^n$. We will say that $(\Omega,\mu)$ is a {\it spectral pair} if (1) for each $f\in L^1(\Omega)\cap L^2(\Omega)$, the continuous function $\lambda\rightarrow \hat f(\lambda)=({f},{e_\lambda})$ satisfies $\int |\hat f|^2\,d\mu<\infty$, and (2) the map $f\rightarrow \hat f$ of $L^1(\Omega)\cap L^2(\Omega)\subset L^2(\Omega)$ into $L^2(\mu)$ is isometric and has dense range. 
	
	This map, then, extends by continuity to an isometric isomorphism
	$$F:L^2(\Omega)\rightarrow L^2(\mu).$$
	We write $\hat f$ for $Ff$, when $f\in L^2(\Omega)$.
	
	The set $\Omega$ is called {\it spectral} if there is a measure $\mu$ such that $(\Omega,\mu)$ is a spectral pair. When $\Omega$ has finite measure, the two definitions of spectral measure coincide (Corollary \ref{corp4}). 
\end{definition}

Another important contribution of Pedersen was to clarify the connection between spectral sets and a special kind of unitary groups, that act locally as translations (see Definition \ref{defp5}). If the set $\Omega$ is spectral then one can define the unitary group $(U(t))_{t\in\br^n}$ as modulation in the Fourier space, meaning the unitary operators $U(t)$ are diagonalized by the Fourier transform $F$ on $L^2(\Omega)$ (Definition \ref{defp2}). This unitary group acts locally as translations on $L^2(\Omega)$ (Proposition \ref{prp4}). In the classical cases of $\Omega=\br$ and $\Omega=[0,1]$, these unitary groups are usual translations (Example \ref{ex8.1} , Example \ref{ex8.4}). 

There is a one-to-one correspondence between commuting self-adjoint operators and unitary groups given by the Generalized  Stone Theorem \ref{thgs}. In this correspondence, the self-adjoint restrictions $\{H_i\}$ of the partial differential operators $\{D_i\}$ correspond to unitary operators $\{U(t)\}$ with the local translation property in Definition \ref{defp5} (Theorem \ref{thai}). 

Finally, using some improved versions of the Spectral Theorem (Theorem \ref{thcst} and Theorem \ref{thmf}), Pedersen generalizes Fuglede's Theorem, and shows that, for {\it connected} open sets (possibly of infinite measure), the existence of commuting self-adjoint restrictions of the partial differential operators $D_j$ is equivalent to $\Omega$ being spectral. 

\begin{definition}\label{defp2}
	Let $(\Omega,\mu)$ be a spectral pair. Then $F:f\rightarrow \hat f$ is an isometric isomorphism of $L^2(\Omega)$ onto $L^2(\mu)$. We may therefore define a strongly continuous unitary representation $U$ of $\br^n$ on $L^2(\Omega)$, by ``modulation in the frequency space'', meaning multiplication by an exponential function in the Fourier domain, and then conjugating using the Fourier transform $F$. More precisely:
	$$(F(U(t)f))(\lambda)=e^{2\pi it\cdot \lambda}(Ff)(\lambda),\quad(f\in L^2(\Omega), t\in\br^n,\lambda\in\Lambda:=\supp\mu).$$
	
	We say that $U$ is the {\it unitary group} associated to the spectral pair $(\Omega,\mu)$.
	
	Note that $U$ may be rewritten in the form, in the distribution sense:
	\begin{equation}
	\label{eqfe}
	(U(t)f, e_\lambda)=e^{2\pi it\cdot\lambda}(f,e_\lambda),\quad(f\in C_0^\infty(\Omega),\lambda\in\supp{\mu}).
	\end{equation}
	
	In this way, $e_\lambda$ can be considered as a generalized eigenfunction for $U$ corresponding to the eigenvalue $e^{2\pi it\cdot\lambda}$.
	
		By the the Generalized Stone Theorem \ref{thgs}, there exists a unique spectral measure $E$ on $\br^n$ such that 
	$$U(t)=\int_{\br^n}e^{2\pi it\cdot \lambda}\,dE(\lambda),\quad (t\in\br^n).$$
	
	The {\it spectrum} of $U$, denoted $\operatorname*{sp} U$ is by definition the support of the spectral measure $E$, which is the same as the joint spectrum of the infinitesimal generators $(H_1,\dots,H_n)$, see \eqref{eqE} and \eqref{eqprE}. 
	
\end{definition}

\begin{proposition}\label{prp3} \cite[Proposition 1.10]{Ped87}
	If $(U(t))_{t\in\br^n}$ is the unitary group associated to a spectral pair $(\Omega,\mu)$ then the spectrum of $U$ is $\operatorname*{sp}(U)=\operatorname*{supp}{\mu}$. 
\end{proposition}

\begin{proof}
 Let $E$ be the spectral measure associated to $E$ as in the Generalized Stone Theorem \ref{thgs}. 
For $f,g\in L^2(\Omega)$, 
	$$\int_{\br^n}e^{2\pi i t\cdot\lambda}\,d\ip{E(\lambda)f}{g}=\ip{U(t)f}{g}=\ip{FU(t)f}{Fg}=\int_{\br^n}e^{2\pi it\cdot\lambda}(f,e_\lambda)(e_\lambda,g)\,d\mu(\lambda).$$
	Hence, 
 $d\ip{E(\lambda)f}{g}=(f,e_\lambda)(e_\lambda,g)\,d\mu(\lambda)=\hat f(\lambda)\cj{\hat g}(\lambda)\,d\mu(\lambda)$. Since the transformation $f\rightarrow \hat f$ is an isometric isomorphism onto $L^2(\mu)$, this means that the support of the spectral measure $E$ is the support of the measure $\mu$, and therefore $\operatorname*{sp} U =\supp{E}=\supp{\mu}$.
	
\end{proof}

In the case of sets $\Omega$ of finite measure, the two Definitions \ref{defsp} and \ref{defp1} of spectral sets coincide:

\begin{corollary}\label{corp4}\cite[Corollary 1.11]{Ped87}
Let $\Omega$ be a measurable subset of $\br^n$, and let $\mu$ be a measure on $\br^n$. Assume that $\Omega$ has finite measure.   $(\Omega,\mu)$ is a spectral pair, iff $\mu(\{\lambda\})=1$ for all $\lambda\in\operatorname*{supp}{\mu}$, and $\{e_\lambda :\lambda\in\Lambda\}$ is an orthogonal basis for $L^2(\Omega)$.
\end{corollary}

\begin{proof}
	
	Note first that for $\lambda\in \Lambda=\operatorname*{supp}{\mu}$ and $f\in L^2(\Omega)$, using \eqref{eqfe}, 
	$$\ip{U(t)e_\lambda}{f}=\ip{e_\lambda}{U(-t)f}=e^{2\pi it\cdot \lambda}\ip{e_\lambda}{f},$$ and therefore $U(t)e_\lambda=e^{2\pi it\cdot\lambda}e_\lambda$.

	Then, for fixed $\lambda,\gamma$ in $\Lambda$ and variable $t\in\br^n$, we have 
	$$e^{2\pi it\cdot\gamma}\ip{e_\gamma}{e_\lambda}=\ip{U(t)e_\gamma}{e_\lambda}=\ip{e_\gamma}{U(-t)e_\lambda}=e^{2\pi it\cdot\lambda}\ip{e_\gamma}{e_\lambda}.$$
	
	Hence, if $\gamma\neq\lambda$, then $\ip{e_\gamma}{e_\lambda}=0$. 
	
	Also, this implies that $Fe_\lambda(\gamma)=0$ for $\gamma\neq \lambda$. Also $Fe_\lambda(\lambda)=\ip{e_\lambda}{e_\lambda}=m(\Omega)$. Since $F$ is unitary, we must have 
	$$m(\Omega)=\|e_\lambda\|_{L^2(\Omega)}^2=\|Fe_\lambda\|_{L^2(\mu)}^2=m(\Omega)\cdot\mu(\{\lambda\}).$$
	This implies that $\mu(\{\lambda\})=1$, for all $\lambda\in\operatorname*{supp}{\mu}$.
	
	The converse implication is simple. 
	
\end{proof}

We know that the classical Fourier transform on $\br$, transforms translations into modulations. For a spectral pair $(\Omega,\mu)$, we defined the unitary group $U$ as modulation in the frequency space $L^2(\mu)$. Going back to the time space $L^2(\Omega)$, with the Fourier transform $F$, we will see in Proposition \ref{prp4} that the unitary group $U$ acts, at least locally, as translations. We make this precise in the following definition.
\begin{definition}
	\label{defp5}
	Let $\Omega$ be an open subset of $\br^n$. We say that $\Omega$ has {\it the integrability property} if there exists a strongly continuous unitary representation $U$ of $\br^n$ on $L^2(\Omega)$ which satisfies the following condition: for every $x$ in $\Omega$, for all neighborhoods $V$ of $x$ in $\Omega$ and all $\epsilon>0$ such that $V+t\subset\Omega$ for all $t\in\br^n$ with $\|t\|<\epsilon$, we have 
	$$(U(t)f)(y)=f(y+t),\mbox{ for all $y\in V$, $t\in\br^n$ with $\|t\|<\epsilon$ and }f\in L^2(\Omega).$$

\end{definition}

\begin{proposition}\label{prp6}
	\label{prp4}\cite[Proposition 2.1]{Ped87}
Let $\Omega$ be an open subset of $\br^n$. If $\Omega$ is a spectral set then $\Omega$ has the integrability property. 	
\end{proposition}

\begin{proof}
 Let $\mu$ be a measure on $\br^n$ such that $(\Omega,\mu)$ is a spectral pair, and let $U$ be the associated unitary group. Let $x\in\Omega$. Let $V$ be a neighborhood of $x$ and $\epsilon>0$ as in Definition \ref{defp5}. Now pick a continuous function $\varphi$ with support in $V$ and a vector $t\in\br^n$ with $\|t\|<\epsilon$. We have for $f\in L^1(\Omega)\cap L^2(\Omega)$,
 $$\ip{U(t)f}{\varphi}=\ip{FU(t)f}{F\varphi}=\int e^{2\pi it\cdot\lambda}({f},{e_\lambda})({e_\lambda},{\varphi})\,d\mu(\lambda)
 =\int({f},{e_\lambda})({e_\lambda(\cdot+t)},{\varphi})\,d\mu(\lambda)$$
 $$=\int({f},{e_\lambda})({e_\lambda},{\varphi(\cdot-t)})\,d\mu(\lambda)=\ip{Ff}{F(\varphi(\cdot-t))}=\ip{f}{\varphi(\cdot-t)}=\ip{f(\cdot+t)}{\varphi}.$$
 
 Hence $(U(t)f)(y)=f(y+t)$ for $y\in V$ and $\|t\|<\epsilon$, which means that $\Omega$ has the integrability property.

\end{proof}

Thus, spectral sets have the integrability property. The next important theorem shows that the integrability property is equivalent to the existence of commuting self-adjoint restrictions of the differential operators $D_j$. 

\begin{theorem}\label{thai} \cite[Lemma 1]{Jor82},\cite[Proposition 1.2]{Ped87} Let $\Omega$ be an open set in $\br^n$. The following statements are equivalent:
	\begin{enumerate}
		\item There exists $H=(H_1,\dots,H_n)$ a commuting family of self-adjoint restrictions $H_j$ of $D_j$ on $L^2(\Omega)$. 
		\item The set $\Omega$ has the integrability property. 
	\end{enumerate}
	The correspondence $H\leftrightarrow U(t)$ is given by the Generalized Stone Theorem \ref{thgs} (see also Theorem \ref{thstone0}).
	
\end{theorem}

\begin{proof}
	
	Assume (i). The unitary group $(U(t))_{t\in\br^n}$ is defined using the spectral measure $E$ for $H$ (Theorem \ref{thspectral}), by 
	$$U(t)=\int e^{2\pi i\lambda\cdot t}\,d E(\lambda),\quad (t\in\br^n).$$
	We need to prove the integrability property. Let $x\in\Omega$ and let $V$ be a neighborhood of $x$ and $\epsilon>0$ such that $t+V\subset \Omega$ and for all $t\in\br^n$ with $\|t\|<\epsilon$. 
	
	Define $(T(t))_{t\in\br^n}$ to be the usual unitary group of translations on $\br^n$, $(T(t)f)(x)=f(x+t)$, for $x,t\in\br^n$, and $f$ some function on $\br^n$.

	Let $f$ be a function in $C_0^\infty(V)$. We want to show that $U(u)f=T(u)f$ if $\|u\|<\epsilon$.

	Note first that, using Theorem \ref{thstone0}(b),  $\delta_i$ being the canonical vectors in $\br^n$,
	$$\frac{\partial}{\partial x_i}U(t)f=\lim_{s\rightarrow 0}\frac{1}{s}(U(t+s\delta_i)f-U(t)f)=U(t)\lim_{s\rightarrow 0}\frac{1}{s}(U(s\delta_i)f-f)=2\pi iU(t)H_if.$$
	Commuting $U(s\delta_i)$ and $U(t)$ we obtain:
	\begin{equation}\label{eqai1}
		\frac{\partial}{\partial x_i}U(t)f=2\pi iU(t)H_if=2\pi iH_iU(t)f.
	\end{equation}

	Define
	$$U_u(t)=U(tu),\quad T_u(t)=T(tu),\quad\mbox{ for }t\in[0,1].$$
	We have 
	$$\frac{d}{dt}U_u(t)f=\sum_{i=1}^n\frac{\partial}{\partial x_i}U(tu)f\cdot u_i=\sum_{i=1}^n2\pi iu_i U(tu)H_if.$$
	Similarly,
	$$\frac{d}{dt}T_u(t)f=\sum_{i=1}^n\frac{\partial}{\partial x_i}T(tu)f\cdot u_i=\sum_{i=1}^nu_i \frac{\partial}{\partial x_i}T(tu)f.$$
	
	Then $$\frac{d}{dt}U_u(1-t)T_u(t)f=\frac{d}{dt}U_u(1-t)T_u(t)f+U_u(1-t)\frac{d}{dt}T_u(t)f$$
	$$=-\sum_{i=1}^nu_iU_u(1-t)2\pi iH_iT_u(t)f+U_u(1-t)\sum_{i=1}^nu_i\frac{\partial}{\partial x_i}T_u(t)f$$
	$$=U_u(1-t)\left(\sum_{i=1}^nu_i\left(-2\pi iH_i+\frac{\partial}{\partial x_i}\right)\right)T_u(t)f.$$
	But, $T_u(t)f$ is in $C_0^\infty(\Omega)$ since $f\in C_0^\infty(V)$, so $2\pi iH_i T_u(t)f=\frac{\partial}{\partial x_i}T_u(t)f$. Therefore, we obtain that 
	$$\frac{d}{dt}U_u(1-t)T_u(t)f=0.$$
	With the Fundamental Theorem of Calculus we obtain 
	$$0=\int_0^1\frac{d}{dt}U_u(1-t)T_u(t)f\,dt=U_u(0)T_u(1)f-U_u(1)T_t(0)f=T(u)f-U(u)f,$$
	which means that $U(u)f=T(u)f$ when $f$ is in $C_0^\infty(V)$.

	We prove now that for any $f\in L^2(\Omega)$, $(U(u)f)(x)=f(x+u)$ for $x\in V$ and $\|u\|<\epsilon$. 
	
	Let $g\in C_0^\infty(V)$. By the previous argument, $(U(-u)g)(y)=g(y-u)$, for $y\in\Omega$. We have 
	$$\int_V(U(u)f)(x)\cj g(x)\,dx=\int_\Omega (U(u)f)(x)\cj g(x)\,dx=\ip{U(u)f}{g}=\ip{f}{U(-u)g}=\int_{\Omega} f(y)\cj g(y-u)\,dy$$
	$$=\int_{V+u}f(y)\cj g(y-u)\,dy=\int_Vf(x+u)\cj g(x)\,dx.$$
	Since $C_0^\infty(V)$ is dense in $L^2(V)$ it follows that $(U(u)f)(x)=f(x+u)$ for a.e. $x\in V$. This proves (ii).

	Assume now (ii). We have a unitary group $U(t)$ and the integrability property. By the Generalized Stone Theorem \ref{thgs}, and Theorem \ref{thstone}, there are commuting self-adjoint operators $(H_1,\dots, H_n)$ which are the infinitesimal generators of the groups $(U(t\delta_i))_{t\in\br}$. We will prove that $H_i$ is a self-adjoint extension of $\frac{1}{2\pi i}\frac{\partial}{\partial x_i}$. 
	
	Let $f\in C_0^\infty(\Omega)$. Let $x\in\Omega$. By the integrability property, and since $\Omega$ is open, there exists $\epsilon>0$ and a neighborhood $V$ of $x$ such that $(U(t)f)(y)=f(y+t)$ for all $y\in V$ and $t\in \br^n$, with $\|t\|<\epsilon$. Then, with Theorem \ref{thstone0}(b), for any measurable subset $F$ of $V$:
	$$\ip{2\pi iH_if}{\chi_F}=\lim_{t\rightarrow0}\frac{1}{t}\ip{U(t\delta_i)f-f}{\chi_F}=\lim_{t\rightarrow0}\frac1t\int_F((U(t\delta_i)f)(x)-f(x))\,dx$$$$=\lim_{t\rightarrow0}\int_F\frac1t(f(x+t)-f(x))\,dx=\int_F\frac{\partial f}{\partial x_i}(x)\,dx=\ip{\frac{\partial f}{\partial x_i}}{\chi_F}.$$
	
	Since the characteristic functions of the type $\chi_F$ above span $L^2(\Omega)$, it follows that $2\pi i H_if=\frac{\partial f}{\partial x_i}$, which means that $H_i$ is indeed a restriction of $D_i$, and we obtain (i) (see also Remark \ref{rem2.1}).
\end{proof}

\begin{remark}
	The integrability condition can be improved in the connected components of $\Omega$. If $x$ and $x'$ are two points in $\Omega$, then we can step from $x$ to $x'$ in a succession of small steps $x_i\rightarrow x_{i+1}=x_i+t_i$ as in the Definition \ref{defp5}; more precisely, one can find a sequence of points $x=x_1,x_2,\dots,x_p=x'$ in $\Omega$, neighborhoods $V_i$ of $x_i$ and numbers $\epsilon_i>0$ such that $\|x_{i+1}-x_i\|<\epsilon_i$ and  $t_i+V_i\subset\Omega$ for all $\|t_i\|<\epsilon_i$, $i=1,\dots, p-1$. Moreover, we can assume $(x_{i+1}-x_i)+V_i\subset V_{i+1}$ for all $i=1,\dots,p-1$.

	Then, for $y\in V_1$ and for some continuous bounded function $f$ on $\Omega$, we have, using the integrability property,
	$$U(x-x')(y)=U(x_2-x_1)U(x_3-x_2)\dots U(x_p-x_{p-1})f(y)=U(x_3-x_2)\dots U(x_p-x_{p-1})f(y+x_2-x_1)$$
		$$=\dots=U(x_p-x_{p-1})f(y+x_{p-1}-x_1)=f(y+x_p-x_1)=f(y+x'-x).$$ 
	Thus, we have the following:
	\end{remark}

\begin{proposition}\label{prp7}
	If the open set $\Omega$ has the integrability property and $x,x+t$ belong to the same connected component of $\Omega$ then 
	\begin{equation}
		\label{eqp5.0}
			(U(t)f)(x))=f(x+t)\mbox{ for any bounded continuous function $f$ on $\Omega$.}
	\end{equation}
\end{proposition}

Finally, together with Theorem \ref{thai}, we have Pedersen's improvement of Fuglede's Theorem \ref{th2.1}, which shows the equivalence, for connected open sets, of the three categories: (1) spectral (possibly unbounded) sets $\Omega$, (2) commuting self-adjoint restrictions $\{H_i\}$ of the partial differential operators $\{D_i\}$, and (3) unitary groups $\{U(t)\}$ that act locally as translations. 
\begin{theorem}\label{thp5}\cite[Theorem 2.2]{Ped87}
	Let $\Omega$ be an open and connected subset of $\br^n$. Then $\Omega$ has the integrability property if and only if $\Omega$ is a spectral set. 
\end{theorem}

\begin{proof}
	 By Proposition \ref{prp4}, we have to prove only the direct implication. We assume that $\Omega$ has the integrability property; with Theorem \ref{thai}, there exist commuting self-adjoint restrictions $H_j$ of the differential operators $D_j$. First, we use the Complete Spectral Theorem \ref{thcst} to diagonalize the partial differential operators $H_j$. There exists a regular Borel measure $\sigma$ on $\br^n$ and a measurable field of Hilbert spaces $\hat{\mathfrak{H}}$, and an isometric isomorphism  $\mathcal F:L^2(\Omega)\rightarrow \hat{\mathfrak{H}}=\int\hat{\mathfrak{H}}(\lambda)\,d\sigma(\lambda)$, such that 
	\begin{equation}\label{eqp5.1}
		\mathcal F(H_jf)(\lambda)=\lambda_j(\mathcal Ff)(\lambda),
	\end{equation}
	for $\lambda\in\Lambda=\operatorname*{supp}{\sigma}$, and $f$ in the domain $\mathscr D(H_j)$ of $H_j$, $(\lambda=(\lambda_1,\dots,\lambda_n))$. 
	
	Also, there exist measurable vector fields $\{y_k\}$ so that $\{y_k(\lambda): k=1,\dots,\dim\hat{\mathfrak{H}}(\lambda)\}$ is an orthonormal basis for $\hat{\mathfrak{H}}(\lambda)$, for all $\lambda\in\Lambda$, $y_k(\lambda)=0$ for $k>\dim\hat{\mathfrak{H}}(\lambda)$.

	Now, here comes one of the main points in Pedersen's proof. Take a sequence $\{\Omega_j\}$ of bounded open subsets of $\Omega$ with $\cj\Omega_j\subset\Omega_{j+1}$, $j=1,2,\dots$, and $\cup_j \Omega_j=\Omega$. The inclusion of the pre-Hilbert $C_0^\infty(\Omega_j)$ with the inner product from Sobolev space $H^m(\Omega_j)$ (see the Appendix), with $m$ large enough, into $L^2(\Omega)$, is a Hilbert-Schmidt operator, by Theorem \ref{tha2}. The space $\Phi=C_0^\infty(\Omega)$ is an inductive limit of the spaces $C_0^\infty(\Omega_j)$, and therefore we can use the Fundamental Theorem \ref{thmf}! It tells us that, for $\sigma$-almost every $\lambda$, there is a distribution $e_k(\lambda)$ on $\Omega$ such that 
	\begin{equation}
		\label{eqp5.3}
		(\mathcal F\varphi)_k(\lambda):=\ip{\mathcal F\varphi(\lambda)}{y_k(\lambda)}=(\varphi,e_k(\lambda)),\quad(\varphi\in C_0^\infty(\Omega),k=1,\dots,\dim\hat{\mathfrak{H}}(\lambda)).
	\end{equation}
	
	We have, for $\varphi\in C_0^\infty(\Omega)$,
	$$(\mathcal F(H_j\varphi))_k(\lambda)=(H_j\varphi,e_k(\lambda))=(\varphi,\frac{1}{2\pi i}\frac{\partial e_k(\lambda)}{\partial x_j}).$$
	On, the other hand, from \eqref{eqp5.1},
	$$(\mathcal F(H_j\varphi))_k(\lambda)=\lambda_j\mathcal F(\varphi)_k(\lambda)=(\varphi,\lambda_je_k(\lambda)).$$
	This means that $\frac1{2\pi i}\frac{\partial e_k(\lambda)}{\partial x_j}=\lambda_j e_k(\lambda)$, or, equivalently $\frac{\partial}{\partial x_j}(e_{-\lambda}e_k(\lambda))=0$, for $j=1,\dots,n$. This system of differential equations has the solution (since $\Omega$ is connected)
	$$e_k(\lambda)=c_ke_\lambda,\quad (k=1,\dots,\dim\hat{\mathfrak{H}}(\lambda)),$$
	where $c_k$ are constants, that depend on $\lambda$. 
	
	The next step is to prove that the dimension of the space $\hat{\mathfrak{H}}(\lambda)$ is 1, for $\sigma$-a.e. $\lambda$. Let $\Lambda_k:=\{\lambda : \dim\hat{\mathfrak{H}}(\lambda)\geq k\}$. We prove that $\Lambda_2$ is a set of $\sigma$-measure zero. Assume that on the contrary $\sigma(\Lambda_2)\neq 0$, then $y_k\neq 0$ for $k=1,2$. Further, for $\varphi\in C_0^\infty(\Omega)$, 
	$$\ip{\mathcal F\varphi}{y_k}=\int \ip{(\mathcal F\varphi)(\lambda)}{y_k(\lambda)}\,d\sigma(\lambda)=\int (\mathcal F\varphi)_k(\lambda)\,d\sigma(\lambda)$$$$=\int(\varphi,e_k(\lambda))\,d\sigma(\lambda)=\int(\varphi,e_\lambda)\cj c_k\,d\sigma(\lambda).$$
	Hence $c_k$ cannot be identically zero, for $k=1,2$ (because $\mathcal F$ is an isomorphism, and therefore $\mathcal F(C_0^\infty(\Omega))$ is dense in $\hat{\mathfrak H}$), and 
	$$\ip{\mathcal F\varphi}{c_2y_1}=\ip{\mathcal F\varphi}{c_1y_2}.$$
	This implies that $c_2y_1=c_1y_2$ which is impossible, since $y_1(\lambda)$ is orthogonal to $y_2(\lambda)$ on the set $\Lambda_2$ of non-zero measure. The contradiction shows that $\Lambda_2$ has to be a null set, and therefore the multiplicity $\dim\hat{\mathfrak{H}}(\lambda)$ has to be 1 for $\sigma$-almost every $\lambda$. 
	
	Further, we obtain that $e_1(\lambda)=c_\lambda e_\lambda$ and, for $\varphi\in C_0^\infty(\Omega)$, 
	$$(\mathcal F\varphi)(\lambda)=(\mathcal F\varphi)_1(\lambda)=(\varphi,c_\lambda e_\lambda)=\cj{c_\lambda}\hat\varphi(\lambda),\mbox{ for $\sigma$-a.e. $\lambda$}.$$
	This implies that the map $\lambda\rightarrow c_\lambda$ is measurable. Also, if we define the measure $\mu$ by $d\mu(\lambda)=|c_\lambda|^2\,d\sigma(\lambda)$ we have
	
	$$\int|\hat\varphi(\lambda)|^2\,d\mu(\lambda)=\int |c_\lambda|^2|\hat\varphi(\lambda)|^2\,d\sigma(\lambda)=\int |(\mathcal F\varphi)(\lambda)|^2\,d\sigma(\lambda)=\int_\Omega|\varphi(x)|^2\,dx,$$
	since $\mathcal F$ is an isometric isomorphism. Therefore $(\Omega,\mu)$ is a spectral pair.

\end{proof}

\section{Unions of intervals}\label{secun}

More can be said in the case in one dimension, when $\Omega$ is a finite union of intervals; the results in this section can be found in \cite{DJ23,DJ15b}.

Let $\Omega$ be a finite union of intervals in $\br$:

$$\Omega=\bigcup_{i=1}^n(\alpha_i,\beta_i),\mbox{ where }-\infty<\alpha_1<\beta_1<\alpha_2<\beta_2<\dots<\alpha_n<\beta_n<\infty.$$

In this case, the domain of the maximal operator $D$, which is 
$$\mathscr D=\{f\in L^2(\Omega) : \mbox{ The weak/distributional derivative $f'$ is in $L^2(\Omega)$}\}$$
can also be described as 
$$\mathscr D=\{f\in L^2(\Omega): f \mbox{ absolutely continuous on each interval and }f'\in L^2(\Omega)\}.$$
For an absolutely continuous function $f\in\mathscr D$, the side limits $f(\alpha_i+)$ and $f(\beta_i-)$ are well defined, and we denote them by $f(\alpha_i)$ and $f(\beta_i)$, respectively, $i=1,\dots,n$. 
We also use the notation 
$f(\vec\alpha)=(f(\alpha_1),\dots,f(\alpha_n))$, and similarly for $f(\vec\beta)$. 

For $\vec z=(z_1,z_2,\dots,z_n)\in\bc^n$ denote by $E(\vec z)$, the $n\times n$ diagonal matrix with entries $e^{2\pi iz_1}$, $e^{2\pi iz_2}$, $\dots,e^{2\pi iz_n}$.

First of all, the self-adjoint restrictions $H$ of the differential operator $D$ are specified by a certain $n\times n$ ``boundary'' matrix $B$ which relates the boundary values of the functions in the domain by $Bf(\vec\alpha)=f(\vec\beta)$.

\begin{theorem}\label{th7.1}
	A self-adjoint restrictions $H$ of the maximal operator $D$ on $\Omega$ is uniquely determined by an $n\times n$ unitary matrix $B$ (which we call {\it the boundary matrix associated to $H$}), through the condition 
	$$\mathscr D(H)=\{f\in\mathscr D : Bf(\vec\alpha)=f(\vec\beta)\}=:\mathscr{D}_B.$$
	Conversely, any unitary $n\times n$ matrix $B$ determines a self-adjoint restriction $H$ of $D$ with domain $\mathscr D(H)$ as above.

	The spectral measure of $H$ is atomic, supported on the spectrum  
	$$\sigma(H)=\left\{\lambda\in\bc : \det(I-E(\lambda\vec\beta)^{-1}BE(\lambda\vec\alpha))=0\right\}\subseteq\br$$
	which is a discrete unbounded set. For an eigenvalue $\lambda\in\sigma(A)$, the eigenspace has dimension at most $n$, and it consists of functions of the form 
	
	\begin{equation}\label{eq7.1.1}
	f(x)= e^{2\pi  i \lambda x} \sum_{i=1}^n c_i\chi_{(\alpha_i,\beta_i)}(x),\mbox{ where }c=(c_i)_{i=1}^n\in\bc^n\mbox{, and }BE(\lambda\vec\alpha)c=E(\lambda\vec\beta)c.
	\end{equation}
	
\end{theorem}

Next, given the self-adjoint $H$ with boundary matrix $B$, one can define the unitary group $(U(t))_{t\in\br}$ which acts as translation inside the intervals (see point (ii) in the next theorem), and once it reaches a boundary point, it splits with probabilities defined by the matrix $B$ (point (iii) in the next theorem).

\begin{theorem}\label{th7.2}
	Let $H$ be self-adjoint restriction of the maximal operator $D$, with boundary matrix $B$. Let $$U(t)=\exp{(2\pi i t H)},\quad (t\in\br),$$ be the associated one-parameter unitary group.
	\begin{enumerate}
		\item The domain $\D(H)$ is invariant for $U(t)$ for all $t\in\br$, i.e., if $f\in\mathscr D$ with $B f(\vec\alpha)= f(\vec\beta)$, then $U(t)f\in\mathscr D$ with $B(U(t)f)(\vec\alpha)=(U(t)f)(\vec\beta)$.
		\item Fix $i\in\{1,\dots,n\}$ and let $t\in\br$ such that $(\alpha_i,\beta_i)\cap ((\alpha_i,\beta_i)-t)\neq \ty$. Then, for $f\in L^2(\Omega)$, 
		
		\begin{equation}
			(U(t) f)(x)=f(x+t),\mbox{ for a.e. }x\in (\alpha_i,\beta_i)\cap ((\alpha_i,\beta_i)-t).
			\label{equ1}
		\end{equation}

		In particular,
		\begin{equation}
			(U(\beta_i-x) f)(x)=f(\beta_i),\, (U(\alpha_i-x)f)(x)=f(\alpha_i),\mbox{ for }f\in\D_B, x\in (\alpha_i,\beta_i).
			\label{eq4.2.2}
		\end{equation}
		\item For $f\in\mathscr D(H)$, if $x\in (\alpha_i,\beta_i)$ and $t>\beta_i-x$, then
		
		\begin{equation}
			\left[ U(t)f\right](x)=\pi_i\left(B\left[U(t-(\beta_i-x))f\right](\vec\alpha)\right).
			\label{eq4.4.1}
		\end{equation}
		Here $\pi_i:\bc^n\rightarrow\bc$ denotes the projection onto the $i$-th component $\pi (x_1,x_2\dots,x_n)=x_i$.
		
	\end{enumerate}
\end{theorem}

The next question is: when is a union of intervals spectral? The requirement is simple: the constants $c_i$ that define the eigenvectors in \eqref{eq7.1.1} should all be equal, so that the eigenvectors are exponential functions. 
\begin{theorem}\label{th7.3}
	The union of intervals $\Omega$ is spectral if and only if there exists an $n\times n$ unitary matrix $B$ with the property that for each $\lambda\in\br$, the equation $BE(\lambda\vec\alpha)c=E(\lambda\vec\beta)c$, $c\in\bc^n$ has either only the trivial solution $c=0$, or only constant solutions of the form $c=a(1,1,\dots,1)$, $a\in\bc$. In this case a spectrum $\Lambda$ of $\Omega$ is the spectrum of the self-adjoint restriction $H$ of $D$ associated to the boundary matrix $B$, as in Theorem \ref{th7.1}.
	
	If $\Lambda$ is a spectrum for $\Omega$, then the matrix $B$ is uniquely and well-defined by the conditions 
	\begin{equation}
		Be_\lambda(\vec \alpha)=e_\lambda(\vec \beta),\mbox{ for all }\lambda\in\Lambda.
		\label{eqfu4.1}
	\end{equation}
	Moreover 
	\begin{equation}
		\textup{span}\{ e_\lambda(\vec \alpha) : \lambda\in\Lambda\}=\textup{span} \{e_\lambda(\vec\beta):\lambda\in\Lambda\}=\bc^n.
		\label{eqfu4.2}
	\end{equation}
	
	Conversely, if the boundary matrix $B$ has the specified property then a spectrum for $\Omega$ is given by 
	\begin{equation}
		\Lambda=\{\lambda\in\br : Be_\lambda(\vec\alpha)=e_\lambda(\vec\beta)\}.
		\label{eqfu4.3}
	\end{equation}
\end{theorem}

What about the unitary group when the set $\Omega$ is spectral? In this case, the unitary group $\{U(t)\}$ acts as translations not only inside intervals, but also when it jumps from one interval to another. More precisely, for any $t\in\br$, and $f\in L^2(\Omega)$,
$$(U(t)f)(x)=f(x+t)\mbox{ for a.e. }x\in\Omega\cap(\Omega-t).$$
In other words $(U(t)f)(x)=f(x+t)$ when both $x$ and $x+t$ are in $\Omega$. We call such a unitary group, a {\it group of local translations}.

\begin{theorem}\label{th7.4}
	Assume that $\Omega$ is spectral with spectrum $\Lambda$. Then the unitary group $U=U_\Lambda$ associated to $\Lambda$ is a unitary group of local translations. 
	Conversely, if there exists a strongly continuous unitary group of local translations $(U(t))_{t\in\br}$ for $\Omega$, then $\Omega$ is spectral. 
\end{theorem}

\section{Examples}\label{secex}

We start with a simple proposition which allows us to check if a unitary group $U(t)$ is associated to a given spectrum $\Lambda$ of a set $\Omega$. 

\begin{proposition}
	\label{pr8.1}
	Let $\Omega$ be a finite measure spectral set in $\br^n$ with spectrum $\Lambda$. Suppose $U(t\delta_i)$, $t\in\br$, $i=1,\dots,n$ are some bounded operators on $L^2(\Omega)$ with 
	\begin{equation}
		\label{eq8.1.1}
		U(t\delta_i)e_\lambda=e^{2\pi i t\lambda_i}e_\lambda,\mbox{ for all $t\in\br$, $i=1,\dots,n$ and $\lambda\in\Lambda$.}
	\end{equation}
	Here, $\delta_i$, $i=1,\dots, n$, are the canonical vectors in $\br^n$. 
	
	Then $(U(t))_{t\in\br^n}$ is the unitary group associated to $\Lambda$ as in Theorems \ref{thp5} and \ref{thai}.
\end{proposition}

\begin{proof}
	Indeed, if \eqref{eq8.1.1} holds, this means that in the orthogonal basis $\{e_\lambda :\lambda\in\Lambda\}$, the operator $U(t\delta_i)$ is a diagonal operator with entries $\{e^{2\pi it\lambda_i}\}_{\lambda\in\Lambda}$. But these diagonal operators clearly form a strongly continuous unitary group and therefore the same is true for $(U(t))_{t\in\br^n}$.
	
\end{proof}

\begin{example}\label{ex8.1}
	The simplest example of a spectral set is of course the classical Fourier series: the unit interval $[0,1]$ with the spectrum $\bz$. This corresponds to the self-adjoint restriction $H$ of the differential operator $D$ with domain 
$$\mathscr{D}(H)=\{f\in L^2[0,1] : f\mbox{ absolutely continuous}, f'\in L^2[0,1]\mbox{ and } f(0)=f(1)\}.$$

With Theorem \ref{th7.3}, the boundary matrix $B$ is determined by $Be_\lambda(0)=e_\lambda(1)$, for all $\lambda\in\bz$ so $B=1$.

The unitary group $U(t)$ is 
$$(U(t)f)(x)=f((x+t)\mod \bz),\quad(f\in L^2[0,1],t\in\br,x\in[0,1]).$$
This can be checked with Proposition \ref{pr8.1}: for $\lambda\in\bz$, $t\in\br$ and $x\in[0,1]$:
$$(U(t)e_\lambda)(x)=e_\lambda((x+t)\mod \bz)=e_\lambda(x+t)=e^{2\pi i\lambda t}e_\lambda(x).$$

For $[0,1]^n$ with spectrum $\bz^n$, a similar check shows that the unitary group is 
$$(U(t)f)(x)=f((x+t)\mod\bz^n),\quad (f\in L^2[0,1]^n,x\in[0,1]^n,t\in\br^n).$$

\end{example}

\begin{example}\label{ex8.3}
	Consider $\Omega=[0,1]$ with spectrum $\bz+a$ for some $a\in\br$. Note that, for any $\lambda\in\bz$,
	$$e_\lambda(1)=e^{2\pi i (\lambda+a)}=e^{2\pi ia}=e^{2\pi ia}e_\lambda(0).$$
	By Theorem \ref{th7.3}, the boundary matrix is $B=e^{2\pi ia}$. Therefore the domain of the self-adjoint restriction $H$ of the differential operator $D$ is 
	$$\mathscr D=\{ f\in L^2[0,1] : f\mbox{ absolutely continuous }, f'\in L^2[0,1]\mbox{ and} f(1)=e^{2\pi ia}f(0)\}.$$
	
	The group $U(t)$ acts as translation as long as the point stays inside $[0,1)$; once it reaches $1$ the function gets multiplied by the phase $e^{2\pi ia}$ and starts back from $0$. 
	
	For $x\in[0,1)$ and $t\in\br$, let $\{t+x\}=(t+x)\mod\bz\in[0,1)$ be the fractional part of $x+t$, and $\lfloor x+t\rfloor=(x+t)-\{x+t\}$ be the integer part of $x+t$.

	We define 
	$$(U(t)f)(x)=e^{2\pi i a\lfloor x+t\rfloor}f(\{x+t\}),\quad(f\in L^2(\Omega), x\in[0,1],t\in\br).$$
	Think of a point that starts at $x$ and moves in $[0,1]$ during the time $t$. In the beginning, after a few seconds $s$, the point is at $x+s$. Once it reaches $1$, it goes back to $0$, but it also carries the phase $e^{2\pi ia}$. Then it continues to move through the interval $[0,1]$, and when it reaches $1$ again, it moves again back to zero and gets another phase factor $e^{2\pi ia}$, so now it carries a factor $e^{2\pi i\cdot2a}$. And so on, until time $t$ expires. The number $\lfloor x+t\rfloor$ counts how many times the point returns to $0$, and $\{x+t\}$ is the final position in $[0,1]$.
	
	We check with Proposition \ref{pr8.1}, that $U(t)$ is the associated unitary group to the spectrum $\Lambda=\bz+a$. We have, for $\lambda\in\bz$, $x\in[0,1]$, $t\in\br$,
	$$(U(t)e_{\lambda+a})(x)=e^{2\pi ia\lfloor x+t\rfloor}e^{2\pi i(\lambda+a)\{x+t\}}=e^{2\pi i\lambda\lfloor x +t\rfloor}e^{2\pi ia\lfloor x+t\rfloor}e^{2\pi i(\lambda+a)\{x+t\}}$$
	$$=e^{2\pi i(\lambda+a)(\lfloor x+t\rfloor+\{x+t\})}=e^{2\pi i(\lambda+a)(x+t)}=e^{2\pi i(\lambda+a)t}e_{\lambda+a}(x)$$
\end{example}

\begin{example}\label{ex8.4}
	For an example of a set of infinite measure that is part of a spectral pair, the classical example is of course $\Omega=\br$ and $\mu$ the Lebesgue measure on $\br$. The Fourier transform on $\br$ makes $(\Omega,\mu)$ a spectral pair. 
	
	The unitary group in this case is the usual group of translations
	$$(U(t)f)(x)=f(x+t),\quad(f\in L^2(\br),x,t\in\br).$$
	Indeed, we can check \eqref{eqfe}: for $\lambda\in\br$, $f\in L^2(\br)$ and $x,t\in\br$,
	$$(U(t)f,e_\lambda)=\int f(x+t)e^{-2\pi i \lambda x}\,dx=\int f(x)e^{-2\pi i\lambda(x-t)}\,dx=e^{2\pi i\lambda t}(f,e_\lambda).$$
	
	Similarly for the spectral pair $\Omega=\br^n$ and $\mu$ the Lebesgue measure on $\br^n$.

\end{example}

\begin{example}
	Let $\Omega$ be the square $\Omega=[0,1]^2$ and $\Lambda\subset\br^2$, the set of pairs $(\lambda_1,\lambda_2)$ with $\lambda_2$ an even integer and $\lambda_1\in\bz$, or $\lambda_2$ an odd integer and $\lambda_1\in\bz+\frac12$. So 
	$$\Lambda=(\bz\times 2\bz) \bigcup\left(\bz+\frac12\right)\times (2\bz+1).$$
	
	It is easy to check that $\Lambda$ is a spectrum for the square: for $\lambda=(\lambda_1,\lambda_2)$ and $\lambda'=(\lambda_1',\lambda_2')$, 
	$$\ip{e_\lambda}{e_{\lambda'}}=\int_0^1e^{2\pi i(\lambda_1-\lambda_1')x_1}\,dx_1\int_0^1e^{2\pi i(\lambda_2-\lambda_2')x_2}\,dx_2,$$
	So $e_\lambda$ is orthogonal to $e_{\lambda'}$ iff $\lambda_1-\lambda_1'\in\bz$ or $\lambda_2-\lambda_2'\in\bz$. 
	
   To see that $\{e_\lambda :\lambda\in \Lambda\}$ is complete, it is enough to check that, for $k\in\bz^2$, $e_k$ is in their span. For this, we just have to check that $e_{(k_1,2k_2+1)}$ is in their span for $k_1,k_2\in\bz$. 
   But $\bz+\frac12$ is a spectrum for $[0,1]$ and therefore $e_{k_1}$ is in the span of $\{e_{\lambda_1+\frac12} :\lambda_1\in\bz\}$. Consequently,
   $e_{(k_1,2k_2+1)}(x,y)=e_{k_1}(x)e_{2k_2+1}(y)$ is in the span of $\{e_{\lambda_1+\frac12}(x)e_{2k_2+1}(y)=e_{(\lambda_1+\frac12,2k_2+1)}(x,y) : \lambda_1\in\bz\}.$
   
   We define now a unitary group $U(t)$ and we check with Proposition \ref{pr8.1} that it is the one associated to the spectrum $\Lambda$. It will be enough to define $U(0,t_2)$ and $U(t_1,0)$ for $t_1,t_2\in\br$. 
   
   Before this, we make some remarks. Note that, given $\lambda=(\lambda_1,\lambda_2)\in \Lambda$, for the move in the vertical direction, since $\lambda_2\in\bz$, $e_\lambda(x_1,0)=e_{\lambda}(x_1,1)$, so we expect $U(0,t_2)$ to preserve the values when going through the boundary points $(x_1,1)$ back to $(x_1,0)$. Hence, in the vertical direction, $U(0,t_2)$ should act simply as translation $\mod\bz$. 
   
   For the horizontal direction, since $\lambda_1\in\bz$ or $\lambda_1\in\bz+\frac12$, 
   $$e_\lambda(0,x_2)=e^{2\pi i \lambda_2x_2},\quad e_\lambda(1,x_2)=e^{2\pi i(\lambda_1+\lambda_2x_2)}=(-1)^{2\lambda_1}e_{\lambda}(0,x_2).$$
   
   The values might not be the same. We will consider instead, for $x_2\in[0,1)$ the pair $(x_2,s(x_2)) $, where $s(x_2)=(x_2+\frac12)\mod\bz\in[0,1)$. Note also $s(s(x_2))=x_2$.
   We compare the values of a vector on the boundary:
   $$\vec E_\lambda(0,x_2):=\begin{bmatrix}
   	e_\lambda(0,x_2)\\e_\lambda(0,s(x_2))
   \end{bmatrix}=\begin{bmatrix}
   e^{2\pi i\lambda_2x_2}\\e^{2\pi i\lambda_2s(x_2)}
   \end{bmatrix}$$
  
	   $$\vec E_\lambda(1,x_2):=\begin{bmatrix}
		e_\lambda(1,x_2)\\e_\lambda(1,s(x_2))
	\end{bmatrix}=\begin{bmatrix}
		e^{2\pi i(\lambda_1+\lambda_2x_2)}\\e^{2\pi i(\lambda_1+\lambda_2s(x_2))}
	\end{bmatrix}.$$
	
	Let $a=e^{2\pi i\lambda_2x_2}$. If $\lambda_2$ is even then $\lambda_1\in\bz$ and 
		$$\vec E_\lambda(0,x_2)=\begin{bmatrix}
			a\\a
		\end{bmatrix}=\vec E_\lambda(1,x_2).$$
		If $\lambda_2$ is odd then $\lambda_1\in\bz+\frac12$ and 
		$$\vec E_\lambda(0,x_2)=\begin{bmatrix}
	 a\\-a
		\end{bmatrix},\quad \vec E_\lambda(1,x_2)=\begin{bmatrix}
		-a\\a
		\end{bmatrix}.$$
		We are looking for a $2\times 2$ matrix $B$ which maps $\begin{bmatrix}
			a\\a
		\end{bmatrix}$ to itself, and $\begin{bmatrix}
		a\\-a
		\end{bmatrix}$ to $\begin{bmatrix}
		-a\\a
		\end{bmatrix}$. The matrix that does this is the flip matrix $\begin{bmatrix}
		0&1\\1&0
		\end{bmatrix}$.
		
		This gives us the intuition that, for a point moving $(x_1,x_2)$ horizontally with time $(t_1,0)$, our unitary group $U(t_1,0)$ should flip between $x_2$ and $s(x_2)$ each time the boundary is reached, and in the horizontal component it just acts as a translation $\mod \bz$. The integer $\lfloor x_1+t_1\rfloor$ counts how many times the point goes through the boundary. 
		
		We define 
		$$(U(0,t_2)f)(x_1,x_2)=f(x_1,(x_2+t_2)\mod\bz),\quad(f\in L^2([0,1]^2),(x_1,x_2)\in[0,1]^2,t_2\in\br),$$
		$$(U(t_1,0)f)(x_1,x_2)=f((x_1+t_1)\mod \bz, s^{\lfloor x_1+t_1\rfloor}(x_2)),\quad(f\in L^2([0,1]^2),(x_1,x_2)\in[0,1]^2,t_1\in\br).$$
		
		We check the eigenvalue property on $e_\lambda$ with $\lambda=(\lambda_1,\lambda_2)\in\Lambda$. 
		
		$$(U(0,t_2)e_\lambda)(x_1,x_2)=e_\lambda(x_1,(x_2+t_2)\mod \bz)
		=e^{2\pi i(\lambda_1x_1+\lambda_2((x_2+t_2)\mod\bz))}$$$$=e^{2\pi i(\lambda_1x_1+\lambda_2(x_2+t_2))} \mbox{ (since $\lambda_2\in\bz$) }=e^{2\pi i\lambda_2t_2}e_\lambda(x_1,x_2).$$
		
		For the horizontal component
		$$(U(t_1,0)e_\lambda)(x_1,x_2)=e_\lambda((x_1+t_1)\mod\bz,s^{\lfloor x_1+t_1\rfloor }(x_2))$$$$=e^{2\pi i(\lambda_1((x_1+t_1)\mod \bz)+\lambda_2 s^{\lfloor x_1+t_1\rfloor}(x_2))}=:A.$$
		
		If $\lambda_2$ is even, then $\lambda_1\in\bz$ and $\lambda_2 s^{\lfloor x_1+t_1\rfloor}(x_2)=\lambda_2 x_2\mod \bz$, because $s^{\lfloor x_1+t_1\rfloor}(x_2)$ is either $x_2$ or $x_2+\frac12\mod \bz$. Then 
		$$A=e^{2\pi i(\lambda_1(x_1+t_1)+\lambda_2x_2)}=e^{2\pi i\lambda_1t_1}e_\lambda(x_1,x_2).$$
		
		If $\lambda_2$ is odd, then $\lambda_1\in\bz+\frac12$. If $\lfloor x_1+t_1\rfloor$ is even then $s^{\lfloor x_1+t_1\rfloor}(x_2)=x_2$. Also, $\lambda_1\lfloor x_1+t_1\rfloor\in\bz$ so  $\lambda_1((x_1+t_1)\mod \bz) =\lambda_1(x_1+t_1)-\lambda_1\lfloor x_1+t_1\rfloor$. Therefore
		$$A=e^{2\pi i(\lambda_1(x_1+t_1)+\lambda_2x_2)}=e^{2\pi i\lambda_1t_1}e_\lambda(x_1,x_2).$$
		
		If $\lfloor x_1+t_1\rfloor$ is odd, then $s^{\lfloor x_1+t_1\rfloor}(x_2)=s(x_2)=(x_2+\frac12)\mod\bz$. Also $\lambda_1((x_1+t_1)\mod \bz)=\lambda_1(x_1+t_1)-\lambda_1\lfloor x_1+t_1\rfloor =\lambda_1(x_1+t_1)+\frac12+k$ for some integer $k$, because $\lambda_1\in\bz+\frac12$. In addition $\lambda_2s(x_2)=\lambda_2x_2+\frac12+l$ for some integer $l$. Then 
		$$A=e^{2\pi i((\lambda_1(x_1+t_1)+\frac12+k+\lambda_2x_2+\frac12+l))}=e^{2\pi i\lambda_1t_1}e_\lambda(x_1,x_2).$$
		The eigenvalue property is checked, and with Proposition \ref{pr8.1}, we get that $\{U(t_1,t_2)\}$ is the unitary group associated to the spectrum $\Lambda$.
\end{example}

\begin{example}
	Consider now $\Omega=[0,\frac12]\cup[1,\frac32]$ with spectrum $\Lambda=2\bz\cup(2\bz+\frac12)$. 
	
	We can check directly that $\Lambda$ is a spectrum, but it will also follow from the study of the self-adjoint restriction of the differential operator $D$ associated to it. 
	
	We consider the boundary values of the exponential functions $e_\lambda$ with $\lambda\in\Lambda$. 
	
	For $\lambda\in 2\bz$,
	$$e_\lambda(\vec\alpha)=\begin{bmatrix}
		e^{2\pi i\lambda\cdot 0}\\
		e^{2\pi i\lambda\cdot 1}
	\end{bmatrix}=\begin{bmatrix}
	1\\
	1
	\end{bmatrix},\quad e_\lambda(\vec\beta)=\begin{bmatrix}
	e^{2\pi i\lambda\cdot \frac12}\\
	e^{2\pi i\lambda\cdot \frac32}\end{bmatrix}=\begin{bmatrix}
		1\\
		1
	\end{bmatrix}=e_\lambda(\vec\alpha).
	$$
	
	For $\lambda\in 2\bz+\frac12$,
		$$e_\lambda(\vec\alpha)=\begin{bmatrix}
		e^{2\pi i\lambda\cdot 0}\\
		e^{2\pi i\lambda\cdot 1}
	\end{bmatrix}=\begin{bmatrix}
		1\\
		-1
	\end{bmatrix},\quad e_\lambda(\vec\beta)=\begin{bmatrix}
		e^{2\pi i\lambda\cdot \frac12}\\
		e^{2\pi i\lambda\cdot \frac32}\end{bmatrix}=\begin{bmatrix}
		i\\
		-i
	\end{bmatrix}=i e_\lambda(\vec\alpha).
	$$
	
	We need a $2\times 2$ matrix $B$ that maps $\begin{bmatrix}
		1\\1
	\end{bmatrix}$ to itself and maps  $\begin{bmatrix}
	1\\-1
	\end{bmatrix}$ to  $\begin{bmatrix}
	i\\-i
	\end{bmatrix}$. Then 
	$$B \begin{bmatrix}
		1\\0
	\end{bmatrix}=\frac12 B\left(\begin{bmatrix}
	1\\1
	\end{bmatrix}+\begin{bmatrix}
	1\\-1
	\end{bmatrix}\right)=\begin{bmatrix}
	\frac{1+i}2\\\frac{1-i}{2}
	\end{bmatrix},\quad
	B \begin{bmatrix}
		0\\1
	\end{bmatrix}=\frac12 B\left(\begin{bmatrix}
		1\\1
	\end{bmatrix}-\begin{bmatrix}
		1\\-1
	\end{bmatrix}\right)=\begin{bmatrix}
		\frac{1-i}2\\\frac{1+i}{2}
	\end{bmatrix}.$$
	Thus $$B=\begin{bmatrix}
		\frac{1+i}{2}& \frac{1-i}{2}\\
		\frac{1-i}{2}&\frac{1+i}{2}
	\end{bmatrix}.$$
	
	Clearly $B$ is unitary and note also that the vector $e_\lambda(\vec\alpha)$ is an eigenvector for the matrix $B$ with eigenvalue $1$ if $\lambda\in 2\bz$, and with eigenvalue $i$ if $\lambda\in 2\bz+\frac12$. 
	
	We check that the matrix $B$ has the property stated in Theorem \ref{th7.3}: if $c=(c_1,c_2)^T$ is a vector in $\bc^2$, $\lambda\in\br$ and $BE(\lambda\vec\alpha)c=E(\lambda\vec\beta)c$, this means 
	$$B\begin{bmatrix}
		c_1\\
		c_2 e^{2\pi i\lambda}
	\end{bmatrix}=\begin{bmatrix}
	c_1e^{2\pi i\lambda\cdot\frac12}\\
	c_2 e^{2\pi i\lambda\cdot\frac13}
	\end{bmatrix}=e^{2\pi i\lambda\frac12}\begin{bmatrix}
	c_1\\
	c_2 e^{2\pi i\lambda}
	\end{bmatrix}.$$
	Then $\begin{bmatrix}
		c_1\\
		c_2 e^{2\pi i\lambda}
	\end{bmatrix}$ is an eigenvector for $B$ with eigenvalue $e^{2\pi i\lambda\cdot\frac12}$; this in turn implies, with the computations above, that $e^{2\pi i\lambda\cdot\frac12}\in\{1,i\}$ and $c_1=c_2$. Thus, we can use Theorem \ref{th7.3}. The self-adjoint restriction $H$ of the differential operator $D$ has domain 
	$$\mathscr{D} (H)=\left\{f\in L^2(\Omega): f\mbox{ absolutely continuous },f'\in L^2(\Omega)\mbox{ and }B\begin{bmatrix}
		f(0)\\
		f(1)
	\end{bmatrix}=\begin{bmatrix}
	f(\frac12)\\
	f(\frac32)
	\end{bmatrix}\right\}.$$
	
	The spectrum is given by $\lambda\in\br$ with $B e_\lambda(\vec\alpha)=e_\lambda(\vec\beta)$. This means
	$$B\begin{bmatrix}
		1\\ e^{2\pi i\lambda}
	\end{bmatrix}=\begin{bmatrix}
	e^{2\pi i\lambda\cdot\frac12}\\ e^{2\pi i\lambda\cdot\frac32}
	\end{bmatrix}=e^{2\pi i\lambda\cdot\frac12}\begin{bmatrix}
	1\\ e^{2\pi i\lambda}
	\end{bmatrix}.$$
	This shows that the vector $\begin{bmatrix}
		1\\ e^{2\pi i \lambda}
	\end{bmatrix}$ has to be an eigenvector for the matrix $B$ with eigenvalue $e^{2\pi i \lambda\frac12}$. But, as we have seen above, the eigenvalues of $B$ are 1 and $i$ and so the condition is satisfied iff $\lambda$ is in $\Lambda=2\bz\cup(2\bz+\frac12)$, which shows also that $\Lambda$ is a spectrum for $\Omega$.

	We examine now the unitary group $\{U(t)\}$ associated to the spectrum $\Lambda$. Think of the pair of points $x\in[0,\frac12]$ and $x+1\in[\frac12,\frac32]$, and let $t$ vary. For small values of $t$ we just translate the points to $(x+t,x+1+t)$; once we reach the boundary points $(\frac12,\frac32)$, the matrix $B$ has to be applied (as in Theorem \ref{th7.2}(c)). The integer $\lfloor x+t\rfloor_{\frac12}:=2((x+t)-((x+t)\mod\frac12\bz))$ counts how many times one has to go through a boundary point. 
	
	For a function $f$ in $L^2(\Omega)$, we can think of $f$ as having two components corresponding to the two intervals, $f\leftrightarrow (f(x),f(x+1))^T$, $x\in[0,\frac12]$. We define
	$$U(t)\begin{bmatrix}
		f(x)\\ f(x+1)
	\end{bmatrix}= B^{\lfloor x+t\rfloor_{\frac12}}\begin{bmatrix}
	f\left((x+t)\mod\frac12\bz\right)\\
	f\left(((x+t)\mod\frac12\bz)+1\right)
	\end{bmatrix},\quad(x\in[0,\frac12]).$$
	
	To see that this is indeed the unitary group associated to the self-adjoint operator $H$, we just need to check that $U(t)e_\lambda=e^{2\pi i\lambda t}e_\lambda$ for $\lambda\in\Lambda$ and $t\in\br$. 
	
	We have
	$$U(t)\begin{bmatrix}
		e^{2\pi i\lambda x}\\
		e^{2\pi i\lambda(x+1)}
	\end{bmatrix}=B^{\lfloor x+t\rfloor_{\frac12}}\begin{bmatrix}
	e^{2\pi i \lambda((x+t)\mod\frac12\bz)}\\
	e^{2\pi i \lambda(((x+t)\mod\frac12\bz)+1)}
	\end{bmatrix}=e^{2\pi i \lambda((x+t)\mod\frac12\bz)}B^{\lfloor x+t\rfloor_{\frac12}}\begin{bmatrix}
	1\\e^{2\pi i\lambda}
	\end{bmatrix}.$$
	
	Recall that $\begin{bmatrix}
		1\\e^{2\pi i\lambda}
	\end{bmatrix}$ is an eigenvector for $B$ with eigenvalue $e^{2\pi i \lambda\cdot\frac12}$, and therefore $$B^{\lfloor x+t\rfloor_{\frac12}}\begin{bmatrix}
	1\\e^{2\pi i\lambda}
	\end{bmatrix}=e^{2\pi i\lambda\frac12\lfloor x+t \rfloor_{\frac12}}\begin{bmatrix}
	1\\e^{2\pi i\lambda}
	\end{bmatrix}.$$
	Note also that $(x+t)\mod\frac12\bz+\frac12\lfloor x+t \rfloor_{\frac12}=x+t$.
	Hence, we obtain further 
		$$U(t)\begin{bmatrix}
		e^{2\pi i\lambda x}\\
		e^{2\pi i\lambda(x+1)}
	\end{bmatrix}=e^{2\pi i \lambda((x+t)\mod\frac12\bz+\frac12\lfloor x+t \rfloor_{\frac12})}\begin{bmatrix}
	1\\e^{2\pi i\lambda}\end{bmatrix}=e^{2\pi i\lambda(x+t)}\begin{bmatrix}
		1\\e^{2\pi i\lambda}\end{bmatrix}=e^{2\pi i\lambda t}\begin{bmatrix}
		e^{2\pi i\lambda x}\\e^{2\pi i\lambda(x+1)}\end{bmatrix}.$$
		Consequently,  $U(t)e_\lambda=e^{2\pi i\lambda t}e_\lambda$, and $\{U(t)\}$ is the unitary group associated to the self-adjoint operator $H$.
	
\end{example}

\section{Appendix: Spectral Theorems and some fundamental inequalities}\label{secap}

 For the benefit of readers, we include in this appendix some of the fundamental results used by Fuglede and Pedersen in their proofs; we also add the following citations for details on unbounded operators and their extensions and for the part of the theory of unitary representations we need below, especially the Stone Theorem. For supplements on unbounded operators, and extensions of the Spectral Theorem, and associated spectral representations, we point to \cite{Con90, DS88, Nel70, Nel69, ReSi80}, and for the relevant parts from the theory of unitary representations \cite{Mac62,Mac60, Jor91, Jor76}.

\subsection{Spectral Theorems, Stone's Theorem}

\begin{definition}\label{defspectral}
	If $X$ is a set, $\mathcal B$ is a $\sigma$-algebra of subsets of $X$, and $\mathfrak H$ is a Hilbert space, a {\it spectral measure/spectral resolution} for $(X,\mathcal B,\mathfrak H)$ is a function $E:\mathcal B\rightarrow\mathscr B(\mathfrak H)$ such that:
	\begin{itemize}
		\item[(a)] for each $\Delta$ in $\mathcal B$, $E(\Delta)$ is a projection;
		\item[(b)] $E(\ty)=0$ and $E(X)=I_{\mathfrak H}$;
		\item[(c)] $E(\Delta_1\cap\Delta_2)=E(\Delta_1)E(\Delta_2)$ for $\Delta_1$ and $\Delta_2$ in $\mathcal B$;
		\item[(d)] if $\{\Delta_n\}_{n=1}^\infty$ are pairwise disjoint sets from $\mathcal B$, then 
		$$E\left(\bigcup_{n=1}^\infty\Delta_n\right)=\sum_{n=1}^\infty E(\Delta_n).$$
	\end{itemize}
	
\end{definition}

\begin{theorem}\label{thspectral}{\bf Spectral Theorem for unbounded self-adjoint operators} \cite[Theorem 4.11,page 323]{Con90} If $A$ is a self-adjoint operator on $\mathfrak H$, then there exists a unique spectral measure $E$ defined on the Borel subsets of $\br$ such that 
	\begin{itemize}
		\item[(a)] $A=\int x\,dE(x)$;
		\item[(b)] $E(\Delta)=0$ if $\Delta\cap\sigma(A)=\ty$;
		\item[(c)] if $U$ is an open subset of $\br$ and $U\cap\sigma(A)\neq\ty$, then $E(U)\neq 0$;
		\item[(d)] if $B\in\mathscr B(H)$ such that $BA\subseteq AB$, then $B(\int\phi\,dE)\subseteq (\int \phi\,dE)B$ for every Borel function $\phi$ on $\br$. 
	\end{itemize}
	
\end{theorem}

\begin{remark}\label{remspec}
	
	Note that, starting with a fixed self-adjoint operator $A$ in a Hilbert space $\mathfrak H$, the corresponding projection valued measure $E$ then induces a spectral representation for $A$  via a system of scalar measures indexed by the vectors in $H$. These are defined for $h,k\in \mathfrak H$ by 
	\begin{equation}
		E_{h,k}(\Delta)=\ip{E(\Delta)h}{k},\mbox{ for all Borel sets }\Delta.
		\label{eqremspec0}
	\end{equation}

	Then the following three conclusions follow immediately  (see \cite[Proposition 4.7, page 321]{Con90}):
	\begin{enumerate}
		\item If $\varphi:\br\rightarrow \bc$ is measurable, then the operator 
		\begin{equation}
			\varphi(A):=\int_{\br}\varphi(x)\,dE(x)
			\label{eqremspec1}
		\end{equation}	
		is well defined in $\mathfrak H$ and normal. 
		\item The dense domain of $\varphi(A)$, denoted $\D(\varphi(A))$, consists of all vectors $h\in\mathfrak H$ such that 
		\begin{equation}
			\int_\br |\varphi(x)|^2\,dE_{h,h}<\infty,
			\label{eqremspec2}
		\end{equation}
		\item For all $h\in \D(\varphi(A))$, 
		\begin{equation}
			\|\varphi(A)h\|^2=\int_{\br}|\varphi(x)|^2\,dE_{h,h}.
			\label{eqremspec3}
		\end{equation}
	\end{enumerate}
	
\end{remark}

\begin{definition}\label{defstone}
	A {\it strongly continuous unitary group} is a function $U:\br^n\rightarrow\mathscr B(\mathfrak H)$ such that for all $s$ and $t$ in $\br^n$:
	\begin{itemize}
		\item[(a)] $U(t)$ is a unitary operator;
		\item[(b)] $U(s+t)=U(s)U(t)$ for all $s,t\in\br$;
		\item[(c)] if $h\in \mathfrak H$ and $t_0\in\br^n$ then $U(t)h\rightarrow U(t_0)h$ as $t\rightarrow t_0$.
	\end{itemize}
	
\end{definition}

\begin{theorem}\label{thstone0}\cite[Theorem 5.1, page 327]{Con90}
	Let $A$ be a self-adjoint operator on $\mathfrak H$ and let $E$ on $(X,\mathcal B,\mathfrak H)$ be its spectral measure. Define 
	$$U(t)=\exp(2\pi it A)=\int e^{2\pi itx}\,d E(x),\quad (t\in\br).$$
	Then 
	\begin{itemize}
		\item[(a)] $(U(t))_{t\in\br}$ is a strongly continuous unitary group;
		\item[(b)] if $h\in\mathscr D(A)$, then 
		$$\lim_{t\rightarrow 0}\frac{1}{t}(U(t)h-h)=2\pi i Ah;$$
		\item[(c)] if $h\in\mathfrak H$ and $\lim_{t\rightarrow 0}\frac1t(U(t)h-h)$ exists, then $h\in\mathscr D(A)$. Consequently, $\D(A)$ is invariant under $U(t)$.
	\end{itemize}
	
\end{theorem}

\begin{theorem}\label{thstone} {\bf Stone's Theorem} \cite[Theorem 5.6, page 330]{Con90}.
	If $U$ is a strongly continuous one parameter unitary group, then there exists a self-adjoint operator $A$ such that $U(t)=\exp(2\pi itA)$, $t\in \br$, and conversely. The self-adjoint operator $A$ is called the infinitesimal generator for $U$. 
	
\end{theorem}

\begin{definition}\label{defcosp}
If we have a family of self-adjoint operators $(A_\beta)_{\beta\in B}$ on a separable Hilbert space $\mathfrak H$, each one will have a spectral measure $E_\beta$ on the Borel subsets of $\br$. We say that the operators $(A_\beta)_{\beta\in B}$ {\it commute in the sense of commuting spectral measure} if for all $\beta_1,\beta_2\in B$ and all Borel subsets $\Delta_1$ and $\Delta_2$ of $\br$, the spectral projections $E_{\beta_1}(\Delta_1)$ and $E_{\beta_2}(\Delta_2)$ commute. 
\end{definition}

One way to look at the Fuglede question for a given dimension $n > 1$, relative to $L^2(\Omega)$, for a fixed measurable set $\Omega$ in $\br^n$, is via the distinction between the following two settings: 1) a system of $n$ unitary one-parameter groups in $L^2(\Omega)$, one for each of the $n$ coordinate directions, vs 2) a single unitary representation of $\br^n$, also acting on $L^2(\Omega)$. Then via Stone's Theorem, the unitary one-parameter groups in 1) may be specified by a corresponding system of $n$ projection valued measures, each defined on the one-dimensional Borel sets; - as opposed to 2) where one must instead specify a single projection valued measure, but now defined on the Borel sets in $\br^n$. To link 1) and 2) we must have the system of projection valued measures from 1) be commutative. Then the corresponding product measure formed from 1) will then be a solution to  2), and vice versa. This is realized also by the Generalized Stone Theorem \ref{thgs}.

\begin{theorem}\label{thcst}{\bf The Complete Spectral Theorem} \cite[page  205, page 333]{Mau67}, \cite[Chapter 6]{Nel69}, \cite{vN49}.
Let $(A_\beta)_{\beta\in B}$ be a commutative family of self-adjoint operators in a separable Hilbert space $\mathfrak H$ (commutative in the sense of commuting spectral projections). Then there is 
\begin{enumerate}
	\item[1.] a direct integral 
	$$\hat{\mathfrak H}=\int_\Lambda\hat{\mathfrak H}(\lambda)\,d\mu(\lambda),$$
	which is a certain family of square integrable vector fields
	\begin{equation}\label{eqst1}
	\Lambda\ni\lambda\rightarrow \hat u(\lambda)\in\hat{\mathfrak H}(\lambda)
	\end{equation}

	on a locally compact space $\Lambda$, the scalar product in $\hat{\mathfrak{H}}$ being
	\begin{equation}\label{eqst2}
	\ip{\hat u}{\hat v}_{\hat{\mathfrak{H}}}=\int_\Lambda \ip{\hat u(\lambda)}{\hat v(\lambda)}_{\hat{\mathfrak H} (\lambda)}\,d\mu(\lambda)=\int_\Lambda\sum_{k=1}^{\dim \hat{\mathfrak{H}}(\lambda)}\hat u_k(\lambda)\cj{\hat v_k(\lambda)}\,d\mu(\lambda);
	\end{equation}
	
	There is a sequence of measurable vector fields $\{y_k\}$, $y_k(\lambda)\in\hat{\mathfrak H}(\lambda)$, for all $\lambda\in\Lambda$, such that $\{y_k(\lambda): k=1,\dots,\dim\hat{\mathfrak H}(\lambda)\}$ is an orthonormal basis for $\hat{\mathfrak H}(\lambda)$, $y_k(\lambda)=0$ for $k>\dim\mathfrak{H}(\lambda)$. The coordinates $\hat u_k(\lambda)$ are given by $\hat u_k(\lambda)=\ip{\hat u(\lambda)}{y_k(\lambda)}_{\hat{\mathfrak H}(\lambda)}$, for all $\lambda\in\Lambda$.
	\item[2.] a unitary map
	$$\mathfrak{H}\ni u\rightarrow \mathcal Fu=\hat u\in\hat{\mathfrak{H}},$$
 which carries all operators $A_\beta$ simultaneously into diagonal operators 
\begin{equation}\label{eqst3}
	(\mathcal FA_\beta u)_k(\lambda)=\hat A_\beta(\lambda)\hat u_k(\lambda),\quad k=1,2,\dots,\dim \hat{\mathfrak H}(\lambda),
\end{equation}
	
	where $\hat A_\beta(\lambda)\in \bc$. 
	
\end{enumerate}

The function $\hat A_\beta(\cdot)$ maps the set $\Lambda$ onto the spectrum of the operator $A_\beta$
\begin{equation}\label{eqst4}
\sigma(A_\beta)=\{\hat A_\beta(\lambda) : \lambda\in\Lambda.\}
\end{equation}

\end{theorem}

\begin{theorem}\label{thgs} {\bf Generalized Stone's Theorem} \cite[Theorem VIII.12, page 270]{ReSi80}. 
	Let $(U(t))_{t\in\br^n}$ be a strongly continuous unitary group on a separable Hilbert space $\mathfrak{H}$. Let $\mathscr D$ be the set of finite linear combinations of vectors of the form 
	$$\varphi_f=\int_{\br^n}f(t)U(t)\varphi\,dt,\quad (\varphi\in\mathfrak H, f\in C_0^\infty(\br^n)).$$
	Then $\mathscr D$ is a domain of essential self-adjointness (in the sense of the existence o a unique self-adjoint extension) for each of the infinitesimal generators $H_j$ of the one-parameter subgroups $U(0,0,\dots,y_j,\dots,0)$, each operator $H_j$ maps $\mathscr D$ into itself, and the operators $H_j$ commute $j=1,\dots,n$ (in the sense of commuting spectral measures). Furthermore, for the joint spectral measure $E$ of $(H_1,\dots, H_n)$  on $\br^n$, 
	$$\ip{U(t)\varphi}{\psi}=\int_{\br^n}e^{2\pi i t\cdot \lambda}\,d\ip{E(\lambda)\varphi}{\psi},\quad\mbox{for all }\varphi,\psi\in\mathfrak H.$$
\end{theorem}

\subsection{Some fundamental inequalities}

Let $\Omega$ be an open subset of $\br^n$, and $l\geq 0$. We denote by $H^l(\Omega)$ the Hilbert space of functions on $\Omega$ that have all weak derivatives up to order $l$ in $L^2(\Omega)$
$$H^l(\Omega)=\left\{ f:\Omega\rightarrow\bc : D^\alpha f \in L^2(\Omega),\mbox{ for }|\alpha|\leq l\}\right\},$$
with the inner product
$$\ip{f}{g}_l=\sum_{|\alpha|\leq l}\int_\Omega D^\alpha f\cdot \cj{D^\alpha g}.$$
(Here $\alpha$ is a multi-index $\alpha=(i_1,\dots,i_k)$ with $i_1,\dots,i_k\in\{1,\dots,n\}$ and $|\alpha|=k$ is its length).

Since functions in $H^l(\Omega)$ can be approximated in $\|\cdot\|_l$-norm by $C^\infty$-functions, by the Meyers-Serrin Theorem, the Hilbert space $H^l(\Omega)$ can be regarded as the completion of the space $C^\infty(\Omega)$ under the inner product $\ip{\cdot}{\cdot}_l$. See \cite[Theorem 3.17. page 67]{AdFo03} or \cite[Theorem 2, page 251]{Ev10}.

We denote by $H_0^l(\Omega)$, the closure of $C_0^\infty(\Omega)$ in $H^l(\Omega)$.

\begin{definition}\label{defnik}
	
	We used in Section \ref{secfu} one of the equivalent forms for the definition of a Nikodym region, see Fuglede \cite[Section 3]{Fug74}, Deny-Lions \cite[p. 328 ff.]{DL45}: a connected open set $\Omega\subset\br^n$ is called a {\it Nikodym region} if $m(\Omega)<\infty$ and there is a finite constant $C=C(\Omega)$ such that, for $u\in L^2(\Omega)$ with $D_1u,\dots, D_nu\in L^2(\Omega)$, 
	
	\begin{equation}
		\left\|u-\frac1{m(\Omega)}\int_{\Omega} u\,dm\right\|^2\leq C(\Omega)\sum_{j=1}^n\|D_ju\|^2.
		\label{eqPoincare}
	\end{equation}
	
	The inequality \eqref{eqPoincare} is the {\it Poincar\' e inequality}, and the smallest possible value of the constant $C(\Omega)$ is called the {\it Poincar\' e constant} for $\Omega$. 
	
	Fuglede mentions in the footnote of page 106 in \cite{Fug74} that it suffice to let $u$ range over $ C^\infty(\Omega)\cap \bigcap_{j=1}^n\mathscr D (D_j)$, and refers to \cite[Chapter II, Section 2]{Fug60}; the Meyers-Serrin Theorem can be used to approximate by $C^\infty(\Omega)$-functions. 
	
	If we apply \eqref{eqPoincare} to $ue_{-\lambda}$, for some $\lambda\in\br^n$, we obtain 
	$$\left\|ue_{-\lambda}-\frac{1}{m(\Omega)}\int_{\Omega}ue_{-\lambda}\,dm\right\|^2\leq C(\Omega)\sum_{j=1}^n\left\|(D_j u)e_{-\lambda}-\lambda_j ue_{-\lambda}\right\|^2.$$
	and this can be rewritten as
	\begin{equation}
		\left\|u-\frac{1}{m(\Omega)}\ip{u}{e_\lambda}e_\lambda\right\|^2\leq C(\Omega) \sum_{j=1}^n\left\|D_ju-\lambda_j u\right\|^2.
		\label{eqPoincare2}
	\end{equation}
	
	Any bounded open connected set with $C^1$ boundary is a Nikodym region \cite[p. 275]{Ev10}; any bounded star-shaped open set is a Nikodym region \cite[p. 332]{DL45}.
	
\end{definition}

\begin{theorem}\label{tha1}(The Second Ehrling-Sobolev Inequality) \cite[Chapter XIV, Section 3; (2.1) page 336, page 337]{Mau67}. 
	Let $\Omega$ be a bounded region in $\br^n$ and let $k\geq 0$ and $m\geq [n/2]+1$. Then there exists a constant, that depends only on $\Omega$, $k$ and $m$, and not on $\varphi$ nor $x$, such that 
	$$|D^\alpha\varphi(x)|\leq C\|\varphi\|_{m+k},\quad (|\alpha|\leq k, x\in\Omega,\varphi\in C_0^\infty(\Omega)).$$

\end{theorem}

\begin{proof}
	
	Let $\varphi$ be a function in $C_0^\infty(\Omega)$. We can extend $\varphi$ by $0$ outside $\Omega$ so that $\varphi$ is also in $C_0^\infty(\br^n)$. Let $x\in\Omega$ and $\xi\in\br^n$ with $\|\xi\|=1$. We define the function 
	$$\psi(t)=\varphi(x+t\xi),\quad(0\leq t\leq 1),$$
	and we apply the Cauchy Integral Remainder form for the Taylor polynomial of $\varphi$ at $t$, $0\leq t\leq 1$:
	$$\varphi(x)=\psi(0)=\sum_{l=0}^{m-1}\frac{1}{l!}\psi^{(l)}(t)(-t)^l+\frac{1}{(m-1)!}\int_t^0\psi^{(m)}(\tau)(0-\tau)^{m-1}\,d\tau.$$
	Throughout the proof we will need to find upper bounds for quantities, and they will involve various constants; we will incorporate them with notations $C_1,C_2, c_1,c_2,\dots$ etc. The constants depend only on $\Omega$, $m$ and $n$, not on $\varphi$, nor $x$. 
	
	With the Schwarz inequality we obtain 
	$$|\varphi(x)|^2\leq C_1\left(\sum_{l=0}^{m-1}|\psi^{(l)}(t)|^2 t^{2l}+\left|\int_0^t\tau^{m-1}\psi^{(m)}(\tau)\,d\tau\right|^2\right).$$
	
 From the Schwarz inequality 
	$$\left|\int_0^t\tau^{m-1}\psi^{(m)}(\tau)\,d\tau\right|^2=\left|\int_0^t\tau^{m-1-\frac{n-1}2}\tau^{\frac{n-1}{2}}\psi^{(m)}(\tau)\,d\tau\right|^2$$
	$$\leq \int_0^t\tau^{2m-2-n+1}\,d\tau \cdot \int_0^t\tau^{n-1}|\psi^{(m)}(\tau)|^2\,d\tau=\frac{1}{2m-n}t^{2m-n}\cdot \int_0^t\tau^{n-1}|\psi^{(m)}(\tau)|^2\,d\tau.$$
	Therefore, we obtain further, using also the fact that $0\leq t\leq 1$,  
	$$|\varphi(x)|^2\leq C_2\left(\sum_{l=0}^{m-1}|\psi^{(l)}(t)|^2+\int_0^t\tau^{n-1}|\psi^{(m)}(\tau)|^2\,d\tau\right).$$
	
	Note that for $0\leq l\leq m$, 
	$$\psi^{(l)}(t)=\sum_{i_1,i_2,\dots,i_l=1}^n\frac{\partial^l\varphi(x+t\xi)}{\partial x_{i_1}\partial x_{i_2}\dots\partial x_{i_l}}\xi_{i_1}\xi_{i_2}\dots\xi_{i_l}.$$
	Now using the fact that $\|\xi\|=1$ and one more Schwarz inequality, we have 
	$$|\psi^{(l)}(t)|^2\leq c_l\left(\sum_{|\alpha|=l}|D^\alpha\varphi(x+t\xi)|^2\right).$$
	Plugging this into the previous inequality, we get that
	$$|\varphi(x)|^2\leq C_3\left(\sum_{|\alpha|\leq m-1}|D^{\alpha}\varphi(x+t\xi)|^2+\sum_{|\alpha|=m}\int_0^t\tau^{n-1}|D^{\alpha}\varphi(x+\tau\xi)|^2\,d\tau\right).$$
	
	Next we integrate on the unit ball in $\br^n$, and use spherical coordinates $0\leq t\leq 1$, $\xi$ on the sphere $\mathbb S^{n-1}=\{\xi\in\br^n: \|\xi\|=1\}$, and therefore $y=x+t\xi$ covers the ball $B(x,1)=\{y\in \br^n :\|y-x\|\leq1\}$. We denote by $v_n$, the volume of the unit ball, $\sigma_{n-1}$ is the surface measure on the sphere $\mathbb S^{n-1}$. Recall that $t^{n-1}$ is the Jacobian of the change of coordinates. So we multiply the previous inequality by $t^{n-1}$ and then integrate $\int_0^1\int_{\mathbb S^{n-1}}\dots\,d\sigma_{n-1}(\xi)\,dt$.
	\begin{multline*}
	|\varphi(x)|^2v_n\leq C_3\left(\sum_{|\alpha|\leq m-1}\int_{B(x,1)}|D^\alpha\varphi(y)|^2\,dy+\right.\\+\left.\sum_{|\alpha|=m}\int_0^1\int_{\mathbb S^{n-1}}t^{n-1}\int_0^t\tau^{n-1}|D^\alpha\varphi(x+\tau\xi)|^2\,d\tau\,d\sigma_{n-1}(\xi)\,dt\right).
	\end{multline*}
	For the last term we interchange the order of integration and we get that it is equal to 
	$$\int_0^1\int_{\mathbb S^{n-1}}\tau^{n-1}|D^\alpha(x+\tau\xi)|^2\int_\tau^1t^{n-1}\,dt\,d\sigma_{n-1}(\xi)\,d\tau\leq C_4\int_{B(x,1)}|D^{\alpha}\varphi(y)|^2\,dy,$$
	since$\int_\tau^1t^{n-1}\,dt\leq C_4$ for some constant.
	
	Therefore, we have, since $\varphi$ and its derivatives are supported on $\Omega$: 
	$$|\varphi(x)|^2v_n\leq C_5\left(\sum_{|\alpha|\leq m}\int_{B(x,1)}|D^{\alpha}\varphi(y)|^2\,dy	\right)\leq C_5\left(\sum_{|\alpha|\leq m}\int_{\Omega}|D^{\alpha}\varphi(y)|^2\,dy	\right),$$
	and this implies that
	$$|\varphi(x)|\leq C_6\|\varphi\|_m.$$
	Applying this inequality to a derivative $D^{\alpha}\varphi$, we obtain the Ehrling-Sobolev inequality. 
\end{proof}


\subsection{Hilbert-Schmidt operators}

For more on Hilbert-Schmidt operators, the reader can consult \cite[Section VI.6]{ReSi80}.

\begin{definition}\label{defhs}
	An operator $A$ from a Hilbert space $\mathfrak{H}$ to a Hilbert space $\mathfrak{H}'$ is called a {\it Hilbert-Schmidt} operator if for some (and hence any) orthonormal basis $\{e_i\}_{i\in I}$ of $\mathfrak{H}$, 
	$$\sum_{i\in I}\|Ae_i\|^2<\infty.$$
	
\end{definition}

\begin{theorem}
	\label{thhs}
	If $A:\mathfrak H\rightarrow\mathfrak{H}'$ is a Hilbert-Schmidt operator between separable Hilbert spaces, then there exists an orthonormal basis $\{u_n\}_{n\in\bn}$ in $\mathfrak H$ and some vectors $\{v_n\}_{n\in\bn}$ in $\mathfrak H'$ such that $\sum_n\|v_n\|^2<\infty$, and 
	$$Au=\sum_{n\in\bn}\ip{u}{u_n}v_n,\quad(u\in\mathfrak H).$$
\end{theorem}

\begin{proof}[Proof sketch.]
	Hilbert-Schmidt operators are compact, and therefore so is $A^*A$. From the spectral decomposition of $A^*A$ we get an orthonormal basis of eigenvectors $\{u_n\}_{n\in\bn}$ in $\mathfrak H$. Define $v_n:=Au_n$ for all $n\in\bn$. Since $A$ is Hilbert-Schmidt, $M:=\sum_n \|v_n\|^2=\sum_n\|Au_n\|^2<\infty$. Therefore the operator 
	$$\tilde Au=\sum_{n\in\bn}\ip{u}{u_n}v_n$$
	is well defined on $\mathfrak H$ and, by the Schwarz inequality and Parseval's identity,
	$$\|Au\|^2\leq \sum_{n\in\bn}|\ip{u}{u_n}|^2\sum_{n\in\bn}\|v_n\|^2=M\|u\|^2.$$
	This means that $\tilde A$ is a bounded linear operator. Also $\tilde Au_n=v_n=Au_n$ for all $n\in\bn$, so $A$ and $\tilde A$ coincide on an orthonormal basis, and therefore they are identical.  
\end{proof}

\begin{theorem}\label{tha2}\cite[Theorem 2, page 337]{Mau67}
	If $\Omega$ is a bounded open set in $\br^n$ then the embedding 
	$$H_0^{m+k}(\Omega)\rightarrow H_0^k(\Omega), \quad k\geq 0, m\geq [n/2]+1,$$
	is a Hilbert-Schmidt transformation. 
	
\end{theorem}

\begin{proof}
	From the Ehrling-Sobolev inequality in Theorem \ref{tha1}, it follows that, for every $|\alpha|\leq k$ and every $x\in\Omega$,  the linear functional $H_0^{m+k}(\Omega)\ni\varphi \rightarrow D^\alpha\varphi(x)\in\bc$ is bounded with a common bound $C$, therefore there exists a vector $b_x^\alpha\in H_0^{m+k}(\Omega)$ with $\|b_x^\alpha\|_{m+k}\leq C$, such that 
	$$D^\alpha\varphi(x)=\ip{\varphi}{b_x^\alpha}_{m+k}.$$
	
	Now let $\{e_i\}$, $i=1,2,\dots$, be an orthonormal basis for $H_0^{m+k}(\Omega)$. To obtain the conclusion of the theorem, we have to show that 
	$$\sum_{i=1}^\infty\|e_i\|_k^2<\infty.$$ 
	We have 
	$$\sum_{i=1}^\infty\|e_i\|_k^2=\sum_{i=1}^\infty\sum_{|\alpha|\leq k}\int_\Omega|D^\alpha e_i(x)|^2\,dx.$$
	However, we have 
	$$C^2\geq \|b_x^\alpha\|_{m+k}^2=\sum_{i=1}^\infty |\ip{b_x^\alpha}{e_i}_{m+k}|^2=\sum_{i=1}^\infty|D^\alpha e_i(x)|^2.$$
	Therefore, we can interchange the sums and the integral and obtain 
	
	$$\sum_{i=1}^\infty\|e_i\|_k^2=\int_{\Omega}\sum_{|\alpha|\leq k}\sum_{i=1}^\infty |D^\alpha e_i(x)|^2\,dx=\int_\Omega\sum_{|\alpha|\leq k}\|b_x^\alpha\|_{m+k}^2\,dx\leq m(\Omega)\cdot\#\{\alpha : |\alpha|\leq k\}\cdot C^2<\infty.$$
	
\end{proof}

\subsection{A stronger form of the Spectral Theorem}
\begin{theorem}\label{thmf}\cite[The Fundamental Theorem, page 334]{Mau67}, \cite{Mau61} In the context of the Complete Spectral Theorem \ref{thcst}, let $\mathfrak{H}_j$, $j=1,2,\dots,$ be linear subspaces of a Hilbert space $\mathfrak{H}$, and suppose that each $\mathfrak{H}_j$ has a pre-Hilbert space structure by means of which the identity embeddings 
$$i_j:\mathfrak H_j\rightarrow \mathfrak{H}$$
are Hilbert-Schmidt transformations. 

Then, for $\mu$-almost every $\lambda\in\Lambda$, the restriction of the transform $\mathcal F$ to $\mathfrak{H}_j$, at the point $\lambda$:
$$\mathcal F_j(\lambda):\mathfrak{H}_j\ni\varphi\rightarrow\hat\varphi(\lambda)\in\hat{\mathfrak{H}}(\lambda)$$
is a Hilbert-Schmidt transformations and therefore continuous, for all $j=1,2\dots$.

Suppose in addition that the inductive limit 
$$\Phi=\operatorname*{lim ind}_{j}\mathfrak{H}_k$$
is dense in $\mathfrak{H}$. Then the transform $\mathcal F\varphi=\hat\varphi$ is given on $\Phi$ by the formula
\begin{equation}
	\label{eqf7}
	\Phi\ni\varphi\rightarrow\ip{\hat\varphi(\lambda)}{y_k(\lambda)}=:\hat\varphi_k(\lambda)=(\varphi,e_k(\lambda)),\quad (k=1,\dots,\dim\hat{\mathfrak H}(\lambda)),
\end{equation}
where $e_k(\lambda)$ are continuous linear functionals on $\Phi$; if $A_\beta\Phi\subset\Phi$, then 
\begin{equation}
	\label{eqf8}
	(A_\beta\varphi,e_k(\lambda))=(\varphi,\hat A_\beta(\lambda) e_k(\lambda)),
\end{equation}
$\mu$-almost everywhere.

\end{theorem}

\begin{proof}
	
	Since the inclusion $i_j$ is a Hilbert-Schmidt Transformation, by Theorem \ref{thhs}, there exists an orthonormal basis $(\eta_i^j)_{i\in I}$ for the completion $\cj{\mathfrak H}_j$, and a family of vectors $(h_i)^j_{i\in I}$ in $\mathfrak H$ with 
	$$\sum_{i\in I}\|h_i^j\|^2<\infty,$$
	such that 
	$$\mathfrak H_j\ni\varphi\rightarrow i_j\varphi=\sum_{i\in I}\ip{\varphi}{\eta_i^j}_{\cj{\mathfrak H}_j}h_i^j\in\mathfrak{H}.$$
	
	The map $\mathcal F$ in the Compete Spectral Theorem is unitary, therefore the composition $\mathcal F_j$:
	
	\begin{equation}\label{eqf10}
	\mathfrak H_j\ni\varphi\rightarrow\hat\varphi=\mathcal F_j\varphi=\sum_{i\in I}\ip{\varphi}{\eta_i^j}_{\cj{\mathfrak{H}}_j}\hat h_i^j\in\hat{\mathfrak{H}},
	\end{equation}
	
	is Hilbert-Schmidt, with 
\begin{equation}\label{eqf11}
	\sum_{i\in I}\|\hat h_i^j\|_{\hat{\mathfrak{H}}}^2=\sum_{i\in I}\| h_i^j\|_{\mathfrak{H}}^2<\infty.
\end{equation}
	
	Note that, from \eqref{eqf10}, we obtain the pointwise equality:  
	$$\mathcal F_j(\lambda)\varphi=\hat\varphi(\lambda)=\sum_{i\in I}\ip{\varphi}{\eta_i^j}_{\cj{\mathfrak H}_j}\hat h_i^j(\lambda),$$
	provided all the convergence requirements are satisfied. This will be true, if we prove that 
	\begin{equation}
		\label{eqf12}
		\sum_{i\in I}\|\hat h_i^j(\lambda)\|_{\hat{\mathfrak H}(\lambda)}^2<\infty\mbox{ for $\mu$-almost every $\lambda\in\Lambda$.}
	\end{equation}
	
	This follows immediately from Fubini's theorem and \eqref{eqf11}:
	
	$$\infty>\sum_{i\in I}\|\hat h_i^j\|_{\hat{\mathfrak{H}}}^2=\sum_{i\in I}\int_\Lambda\|\hat h_i^j(\lambda)\|_{\hat{\mathfrak{H}}(\lambda)}^2\,d\mu(\lambda)=\int_\Lambda\sum_{i\in I}\|\hat h_i^j(\lambda)\|_{\hat{\mathfrak H}(\lambda)}^2\,d\mu(\lambda),$$
			and hence there exists a set $N_j\subset\Lambda$ with $\mu(N_j)=0$, such that $\sum_i\|\hat h_i^j(\lambda)\|_{\hat{\mathfrak{H}}(\lambda)}^2<\infty$ for $\lambda\not\in N_j$. In conclusion, \eqref{eqf12} holds for $\lambda$ outside the measure zero set $\cup_j N_j$.

			Since $\Phi$ is the inductive limit of the spaces $\mathfrak H_j$, and all the maps $\mathcal F_j(\lambda)$ are continuous, it follows that the map $\mathcal F(\lambda):\Phi\ni\varphi\rightarrow\hat\varphi(\lambda)\in\hat{\mathfrak{H}}(\lambda)$ is continuous. This implies that the map $\Phi\ni\varphi\rightarrow\ip{\hat\varphi}{y_k(\lambda)}=\hat\varphi_k(\lambda)$ is a continuous linear functional $e_k(\lambda)$, so 
			$$\hat\varphi_k(\lambda)=(\varphi,e_k(\lambda)),\quad (\varphi\in\Phi).$$
			The equation \eqref{eqf8} follows easily from \eqref{eqst3}.
\end{proof}


\begin{acknowledgements} The authors are pleased to acknowledge multiple discussions (both recent and past), help and inspiration from teachers, colleagues, and collaborators, past and present. The full list is too long to include here. But relating to more recent collaborations, we offer a partial list: we wish to especially thank Steen Pedersen, Alex Iosevich, Nir Lev, Alexandru Tamasan and Feng Tian. And in the case of the second named author, I. Segal, B. Fuglede, David Shale, Jonn Herr, Ka-Sing Lau, Bob Strichartz, Eric Weber, and Daniel Alpay.
\end{acknowledgements}

\bibliography{eframes}

\end{document}